\documentclass[preprint,11pt]{imsart}
\usepackage{amsfonts}
\usepackage{amsthm, amstext, amsmath, amssymb}
\usepackage{mathrsfs}

\RequirePackage[OT1]{fontenc}
\RequirePackage{hyperref}
\allowdisplaybreaks

\numberwithin{equation}{section} 

\newtheorem{theorem}{Theorem}[section]
\newtheorem{lemma}{Lemma}[section]

\newtheorem{proposition}{Proposition}[section]

\theoremstyle{definition} 
\newtheorem{definition}{Definition}[section]
\newtheorem{remark}{Remark}[section]

\newcommand{\twS}{\widetilde{S}}

\def\d{\mathrm{d}}

\newcommand{\Ps}{{\mathsf{P}}}

\newcommand{\E}{{\mathsf{E}}}
\newcommand{\PFA}{\mathsf{PFA}}
\newcommand{\pfa}{\mathsf{PFA}}

\newcommand{\ADD}{{\mathsf{ADD}}}

\newcommand{\Hyp}{{\mathsf{H}}}

\newcommand{\Fc}{{ \mathscr{F}}}
\newcommand{\Bmc}{{ \mathscr{B}}}

\newcommand{\wtT}{\widetilde{T}}

\newcommand{\mb}[1]{\mathbf{#1}} 

\newcommand{\Yb}{{\mb{Y}}}

\newcommand{\mbb}[1]{\mathbb{#1}} 
\def\One{\mathchoice{\rm 1\mskip-4.2mu l}{\rm 1\mskip-4.2mu l}
{\rm 1\mskip-4.6mu l}{\rm 1\mskip-5.2mu l}}
\newcommand\Ind[1]{{\One_{\{#1\}}}}

\newcommand{\classA}{{\mbb{C}}_\alpha}

\newcommand{\mc}[1]{\mathcal{#1}} 

\newcommand{\Nc}{{\mc{N}}}

\newcommand{\Kc}{{\mc{K}}}
\newcommand{\Ac}{{\mc{A}}}

\newcommand{\abs}[1]{\left\vert#1\right\vert}
\newcommand{\set}[1]{\left\{#1\right\}}

\newcommand{\brc}[1]{\left(#1\right)}
\newcommand{\brcs}[1]{\left[#1\right]}

\renewcommand{\le}{\leqslant} 
\renewcommand{\ge}{\geqslant}

\def\arginf{\mathop{\rm arg\,inf}}
\newcommand{\vae}{\varepsilon}

\newcommand{\Ts}{T_{\mathrm{s}}}

\usepackage{color, soul}

\usepackage[numbers]{natbib}

\begin{document}

\begin{frontmatter}
\title{ASYMPTOTIC BAYESIAN THEORY OF QUICKEST CHANGE DETECTION FOR HIDDEN MARKOV MODELS}
\runtitle{Quickest Change Detection for Hidden Markov Models}

\begin{aug}
\author{\large \fnms{Cheng-Der} \snm{Fuh}\ead[label=e1]{cdfuh@cc.ncu.edu.tw}}
\and
\author{\large \fnms{Alexander} \snm{Tartakovsky}\ead[label=e2]{alexg.tartakovsky@gmail.com}}
\runauthor{Fuh and Tartakovsky}

\affiliation{Graduate Institute of Statistics, National Central University, Taoyuan,  Taiwan  and \\
AGT StatConsult, Los Angeles, CA, USA}

\address{Graduate Institute of Statistics   \\
National Central University \\
Taoyuan, Taiwan  \\
\printead{e1}}
\address{AGT StatConsult \\
Los Angeles, CA, USA\\
\printead{e2}}

\end{aug}

\begin{abstract}
In the 1960s, Shiryaev developed a Bayesian theory of change-point detection in the i.i.d.\ case, which was generalized in the beginning of the 2000s by 
Tartakovsky and Veeravalli for general stochastic 
models assuming
a certain stability of the log-likelihood ratio process. Hidden Markov models represent a wide class of stochastic processes that are very useful in a variety of applications.  In this paper, 
we investigate the performance of the Bayesian Shiryaev change-point detection rule for hidden Markov models. We propose a set of regularity conditions under which the Shiryaev procedure 
is first-order asymptotically optimal in a Bayesian context, minimizing moments of the detection delay up to certain order asymptotically as the probability of 
false alarm goes to zero.  
The developed theory for hidden Markov models is based on Markov chain representation for the likelihood ratio and $r$-quick convergence for Markov random walks.
In addition, applying Markov nonlinear renewal theory, we present a high-order asymptotic approximation for the expected delay to detection of the 
Shiryaev detection rule. Asymptotic properties
of another popular change detection rule, the Shiryaev--Roberts rule, is studied as well. Some interesting examples are given for illustration.
\end{abstract}


\end{frontmatter}

\section{Introduction} \label{s:Intro}

Sequential change-point detection problems deal with detecting changes in a state of a random process via 
observations which are obtained sequentially one at a time. If the state is normal, then one wants to 
continue the observations. If the state changes and becomes abnormal, one is interested in 
detecting this change as rapidly as possible. In such a problem, it is always a tradeoff between false 
alarms  and a speed of detection, which have to be balanced in a reasonable way. 
A conventional  criterion is to minimize the expected delay to detection while controlling a risk associated with false detections. An optimality criterion and a solution depend heavily on what is known about the models for the observations and for the change point.

As suggested by Tartakovsky et al.~\cite{TNBbook14}, there are four main problem formulations of a sequential change-point detection problem that differ
by assumptions on the point of change and optimality criteria. In this paper, we are interested in a Bayesian criterion assuming that the change point is random with a given prior distribution.
We would like to find a detection rule that minimizes an average delay to detection, or more generally, higher moments of the detection delay in the class of rules with a given false alarm probability.
At the beginning of the 1960s, Shiryaev~\cite{ShiryaevTPA63} developed a Bayesian change-point detection theory in the i.i.d.\ case when the observations are independent with one distribution 
before the change and with another
distribution after the change and when the prior distribution of the change point is geometric.  Shiryaev found an optimal Bayesian detection rule, which prescribes  
comparing the posterior probability of the change point to a constant threshold. Throughout the paper, we refer to this detection rule  as the Shiryaev rule even in a more general non-i.i.d.\ case.  
Unfortunately, finding an optimal rule in a general 
case of dependent data does not seem feasible. The only known generalization, due to the work of Yakir~\cite{Yakir94}, 
is for the homogeneous finite-state Markov chain. Yakir~\cite{Yakir94} proved that the rule, based on thresholding the posterior probability of the change point with a random threshold that depends on the 
current state of the chain, is optimal. Since in general developing a strictly optimal detection rule is problematic, 
Tartakovsky and Veeravalli~\cite{TartakovskyVeerTVP05} considered the asymptotic problem of minimizing
the average delay to detection as the probability of a false alarm becomes small and proved that the Shiryaev rule is asymptotically optimal as long as the log-likelihood ratio process 
(between the ``change'' and ``no change'' hypotheses) has certain stability properties expressed via the strong law of large numbers and its strengthening into $r$-quick convergence. 
A general Bayesian asymptotic theory of change detection in continuous 
time was developed by Baron and Tartakovsky~\cite{BaronTartakovskySA06}. 

While several examples related to Markov and hidden Markov models were considered in \cite{BaronTartakovskySA06,TartakovskyVeerTVP05}, these are only very particular cases where the main condition on
the $r$-quick convergence of the normalized log-likelihood ratio was verified. Moreover, even these particular examples show that the verification of this condition typically represents a hard task. 
At the same time, 
there is a class of very important stochastic models -- hidden Markov models (HMM) -- that find extraordinary applications in a wide variety of fields such as speech  recognition 
\cite{juang&rabiner93,rabiner:1989}; handwritten recognition \cite{hu:brown:1996,kunda:1989};
computational molecular biology and bioinformatics, including DNA and protein modeling \cite{churchill:1989}; human activity recognition \cite{HAR};
target detection and tracking \cite{Willett,Fusion2010, WSEAS2014}; and modeling, rapid detection and tracking of malicious activity of terrorist groups \cite{RGT2013,Raghavanetal-IEEE2014}, 
to name a few. Our first goal is to focus on this 
class of models  and specify the general results of  Tartakovsky and Veeravalli~\cite{TartakovskyVeerTVP05} for HMMs, finding a set of general conditions under 
which the Shiryaev change-point detection procedure is asymptotically optimal as the probability of false alarm goes to zero. 
Our approach for hidden Markov models is based on Markov chain representation of the likelihood ratio, proposed by \citet{FuhAS03}, and $r$-quick convergence for 
Markov random walks (cf. \citet{FuhZhang00}). In addition, by making use uniform Markov renewal theory and Markov nonlinear renewal theory developed in \cite{FuhAAP04,FuhAS04}, we achieve our second goal by providing a high-order asymptotic approximation to the expected delay to detection of the Shiryaev detection rule. 
We also study asymptotic operating characteristics of the (generalized) Shiryaev--Roberts procedure in the Bayesian context. 

The remainder of the paper is organized as follows. In Section~\ref{s:Overview}, we provide a detailed overview of previous results in change-point detection and give some basic results in the general 
change detection theory for dependent data that are used in subsequent sections for developing a theory for HMMs. In Section~\ref{s:FormulationHMM}, a formulation of the problem for finite state HMMs is 
given. We develop a change-point detection theory for HMMs in Section~\ref{s:FOAO}, where we prove that under a set of quite general conditions on the finite-state HMM, 
the Shiryaev rule is asymptotically optimal
(as the probability of false alarm vanishes), minimizing moments of the detection delay up to a certain order $r\ge 1$. Section~\ref{s:SR} studies the asymptotic performance of the generalized
Shiryaev--Roberts procedure. 
In Section~\ref{s:HOAO}, using Markov nonlinear renewal theory, we provide
higher order asymptotic approximations to the average detection delay of the Shiryaev and Shiryaev--Roberts detection procedures.  
Section~\ref{s:Examples} includes a number of interesting examples that illustrate the general theory. 
Certain useful auxiliary results are given in the Appendix where we also present a simplified version of the Markov nonlinear renewal theory that helps solve our problem.

\section{Overview of the Previous Work and Preliminaries}\label{s:Overview}

Let $\{Y_n\}_{n\ge1}$ denote the series of random observations defined on the 
complete probability space $(\Omega,\Fc, \mathsf{P})$, $\Fc=\vee_{n\ge 1}\Fc_n$, where 
$\Fc_n=\sigma(Y_1,\dots,Y_n)$ is the $\sigma$-algebra generated by the observations. Let $\Ps_\infty$ and $\Ps_0$ be two probability measures defined on this 
probability space. We assume that these measures are mutually locally absolutely continuous, i.e., 
the restrictions of the measures $\Ps_0^{(n)}$ and $\Ps_\infty^{(n)}$ to the $\sigma$-algebras $\Fc_n$, $n \ge 1$ are 
absolutely continuous with respect to each other. 
Let $\Yb_0^n=(Y_0, Y_1,\dots,Y_n)$ denote the vector of the 
$n$ observations $(Y_1,\dots,Y_n)$ with an attached initial value $Y_0$ which is not a real observation but rather  an initialization generated by a ``system'' 
in order to guarantee some desired property of the observed sequence $\{Y_n\}_{n\ge 1}$. Since we will consider asymptotic behavior, this assumption will not affect our resuts.
Let $p_j(\Yb_0^n),~j=\infty,0$ 
denote densities of $\Ps_j^{(n)}$ with respect 
to a $\sigma$-finite measure. Suppose now that the observations $\{Y_n\}_{n\ge0}$ initially follow the 
measure $\Ps_\infty$ (normal regime) and at some point in time $\nu=0,1,\ldots$ something happens and 
they switch to $\Ps_0$ (abnormal regime).
For a fixed $\nu$, the change induces a probability measure $\Ps_\nu$ with density 
$p_\nu(\Yb_0^n)= p_\infty(\Yb_0^{\nu-1}) \cdot p_0(\Yb_{\nu}^n | \Yb_0^{\nu-1})$, which can also be written as
\begin{equation}\label{noniidmodel}
p_\nu(\Yb_0^n)=  \left(\prod_{i=0}^{\nu-1} p_{\infty}(Y_i|\Yb_0^{i-1})\right)\times\left( \prod_{i=\nu}^{n} p_{0}(Y_i|\Yb_0^{i-1})\right),
\end{equation}
where $p_{j}(Y_n|\Yb_0^{n-1})$ stands for the conditional density of $Y_n$ given the past 
history $\Yb_0^{n-1}$. Note that in general the conditional densities $p_0(Y_i | \Yb_0^{i-1})$, $i =\nu, \ldots, n$ may depend 
on the change point $\nu$, which is often the case for hidden Markov models. Model \eqref{noniidmodel} can cover this case as well,  
allowing $p_0^{(\nu)}(Y_i | \Yb_0^{i-1})$ to depend on $\nu$ for $i \ge \nu$. Of course the densities $p_{j}(Y_i|\Yb_0^{i-1})$ may depend on $i$.

In the present paper, we are interested in the Bayesian setting where $\nu$ is a random variable. 
In general, $\nu$ may be dependent on the observations. This situation was discussed in Tartakovsky and Moustakides~\cite{TM2010} and
Tartakovsky et al.~\cite[Sec 6.2.2]{TNBbook14} in detail, and we only summarize the idea here. 
Let $\omega_k$, $k=1,2,\dots$  be probabilities that depend on the 
observations up to time $k$, i.e.,  $\omega_k=\Ps(\nu=k|\Yb_1^k)$ for $k\ge 1$, so that the sequence $\{\omega_k\}$ is $\{\Fc_k\}$-adapted. This allows for a 
very general modeling of the change-point mechanisms, including the case where $\nu$ is  a stopping time 
adapted to the filtration $\{\Fc_n\}_{n\ge 1}$ generated by the observations (see 
Moustakides~\cite{MoustakidesAS08}). However, in the rest of this paper, we limit ourselves to the case 
where $\omega_k$ is deterministic and known. In other words, we follow the Bayesian approach proposed by 
Shiryaev~\cite{ShiryaevTPA63} assuming that $\nu$ is a random variable {\it independent} of the observations with 
a known prior distribution $\omega_k=\Ps(\nu=k)$, $\sum_{k=0}^\infty \omega_k =1$.

A sequential change-point detection rule is a stopping time $T$ adapted to the filtration 
$\{\Fc_n\}_{n\ge 1}$, i.e., $\{T\le n\} \in \Fc_n$, $n\ge1$. To avoid triviality we always assume that  $T\ge 1$ with probability 1. 

Define the probability measure 
$\Ps^{\omega}(\Ac)=\sum_{k=0}^\infty \,\omega_{k}\, \Ps_{k}(\Ac)$ and let $\E^\omega$ 
stand for the expectation with respect to $\Ps^{\omega}$.
The false alarm risk is usually measured by the (weighted) probability of false alarm $\PFA(T)=\Ps^{\omega}(T <\nu)$. Taking into account that $T>0$ with probability\ 1 and $\{T < k\}\in\Fc_{k-1}$, we obtain
\begin{equation} \label{PFAdeforig}
\PFA(T)=\sum_{k=0}^\infty \omega_k \Ps_k(T< k) =\sum_{k=1}^\infty \omega_k  \Ps_\infty(T < k).
\end{equation}
Usually the speed of detection is expressed by the average detection delay (ADD)
\begin{equation} \label{ADDdef}
\ADD(T)=\E^\omega(T-\nu|T \ge \nu)= \frac{\E^\omega(T-\nu)^+}{\Ps^\omega(T \ge \nu)} = \frac{\sum_{k=0}^\infty \omega_k \Ps_\infty(T\ge  k)\E_k(T-k|T\ge  k)}{\sum_{k=0}^\infty \omega_k \Ps_\infty(T\ge  k)} .
\end{equation}
Since for any finite with probability\ $1$ stopping time,
$\sum_{k=0}^\infty \omega_k \Ps_\infty(T \ge k) = 1- \Ps^{\omega}(T < \nu)$, it follows from \eqref{ADDdef} that
 \[
 \ADD(T) =\frac{\sum_{k=0}^\infty \omega_k \Ps_\infty(T\ge  k)\E_k(T-k|T\ge  k)}{1-\PFA(T)} .
 \]

An optimal Bayesian detection scheme is a rule for which the $\ADD$ is minimized in the class of rules  $\classA = \{T: \PFA (T) \le \alpha\}$  with the $\PFA$ constrained to be below a 
given level $\alpha\in(0,1)$, i.e., the optimal change-point detection rule is the stopping time $T_{\rm opt}=\arginf_{T\in\classA}\ADD(T)$.
Shiryaev~\cite{ShiryaevTPA63} considered the case of $\nu$ with a zero-modified geometric distribution
\begin{equation}\label{Geom}
\begin{aligned}
\Ps(\nu=0) = \omega_0, \quad \Ps(\nu=k) =(1-\omega_0)\rho
(1-\rho)^{k-1}, ~ k \ge 1,
\end{aligned}
\end{equation}
where  $\omega_0 \in[0,1)$, $\rho \in (0,1)$.  Note that when $\alpha\ge1-\omega_0$, there is a trivial solution since we can stop at $0$.  
Thus, in the following,  we assume that  $\alpha< 1-\omega_0$.
Shiryaev~\cite{ShiryaevTPA63, ShiryaevBook78} proved that in the i.i.d.~case (i.e., when $p_{j}(Y_n|\Yb_0^{n-1})= p_{j}(Y_n)$ for $j=0,\infty$ in \eqref{noniidmodel}) the optimal Bayesian
detection procedure $T_{\rm opt}=\Ts(h)$  exists and has the form
\begin{equation} \label{Shirst}
\Ts(h)=\inf\{n\ge 1: \Ps(\nu \le  n |\Fc_n) \ge h\},
\end{equation}
where threshold $h=h_\alpha$ is chosen to satisfy $\PFA(\Ts(h_\alpha))=\alpha$.

Consider a general non-i.i.d.\ model \eqref{noniidmodel} and a general, not necessarily geometric prior distribution $\Ps(\nu=0)=\omega_0$, $\Ps(\nu=k) = \omega_k= (1-\omega_0) \tilde{\omega}_k$ for $k\ge 1$, where $\tilde{\omega}_k=\Ps(\nu=k|\nu>0)$, $\sum_{k=1}^\infty \tilde{\omega}_k=1$. Write
$\Lambda_i=p_0(Y_i|\Yb_0^{i-1})/p_\infty(Y_i | \Yb_0^{i-1})$ as the conditional likelihood ratio for the $i$-th sample. 
We take a convention that $\Lambda_0$ can be any random or deterministic number, in particular $1$ if the initial value $Y_0$ is not available, 
i.e., before the observations become available we have no information except for the prior probability $\omega_0$.

Applying the Bayes formula, it is easy to see that
\begin{equation}
\Ps(\nu \le  n|\Fc_n) = \frac{R_{n,\omega}}{R_{n,\omega} + 1},
\label{post_R}
\end{equation}
where
\begin{equation} \label{Rnp}
\begin{aligned}
R_{n,\omega} = \frac{\omega_0 \prod_{i=1}^n \Lambda_i + \sum_{k=1}^n \omega_k \prod_{i=k}^n \Lambda_i}{\Ps(\nu > n)}
=  \sum_{k=0}^n \omega_{k,n} \prod_{i=k}^n \Lambda_i,
\end{aligned}
\end{equation}
with
\begin{equation}\label{wkn}
\omega_{k,n}= \frac{\omega_k}{\sum_{j=n+1}^\infty \omega_j} =  \begin{cases}
\frac{\tilde{\omega}_k}{\sum_{j=n+1}^\infty \tilde{\omega}_j}  & \text{for} ~ k \ge 1, \\
\frac{\omega_0}{(1-\omega_0) \sum_{j=n+1}^\infty \tilde{\omega}_j}  & \text{for} ~ k =0 .
\end{cases}
\end{equation}
If $\Lambda_0=1$,  then $R_{0,\omega} = \omega_{0,0}=\omega_0/(1-\omega_0)$.

It is more convenient to rewrite the stopping time \eqref{Shirst} in terms of the statistic $R_{n,\omega}$, i.e.,
\begin{equation} \label{Shirst1}
T_A = \inf\set{n \ge 1: R_{n,\omega} \ge A},
\end{equation}
where $A=h/(1-h)>{\omega_0}/{(1-\omega_0)}$.  Note that $\prod_{i=k}^n \Lambda_i$ is the likelihood ratio between the hypotheses $\Hyp_k: \nu=k$ 
that a change occurs at the point $k$ and $\Hyp_\infty: \nu=\infty$ (no-change). Therefore, $R_{n,\omega}$ can be interpreted as an average (weighted) likelihood ratio.

Although for general non-i.i.d.\ models no exact optimality properties are available (similar to the i.i.d.~case), there exist asymptotic results.
Define the exponential rate of convergence $c$ of the prior distribution,
\begin{equation} \label{expon}
c =-\lim_{k\to\infty}\frac{\log \Ps(\nu >  k)}{k},\quad \text{$c \ge 0$},
\end{equation}
assuming that the corresponding limit exists. If $c >0$, then the prior distribution has (asymptotically) exponential right tail. If  $c=0$, then this amounts to a heavy tailed distribution. 
Note that for the geometric distribution \eqref{Geom},  $c=-\log (1-\rho)$.

To study asymptotic properties of change-point detection rules  we need the following definition.

\begin{definition}
Let, for $\varepsilon >0$,
\[
\tau_\varepsilon = \sup\{ n \ge 1: |\xi_n| > \varepsilon\} \quad (\sup\{\varnothing\} = 0)
\]
be the last entry time of a sequence $\{\xi_n\}_{n \ge 1}$ into the region $(-\infty, -\varepsilon] \cup [\varepsilon,\infty)$, i.e., the last time after which
$\xi_n$ leaves the interval $[-\varepsilon,\varepsilon]$. It is said that $\xi_n$ converges $r-$quickly to 0 as $n \to\infty$ if $\E (\tau_\varepsilon)^r < \infty$ for all $\varepsilon >0$ and some $r>0$ 
(cf. \cite{Lairquick,TNBbook14}). \end{definition}

The last entry time $\tau_\varepsilon$ plays an important role in the strong law of large numbers (SLLN). Indeed, $\Ps(\tau_\varepsilon < \infty)=1$ for all $\varepsilon >0$ implies that $\xi_n \to 0$ $\Ps$-a.s. as $n \to \infty$. Also, by Lemma 2.2 in \cite{TarSISP98}, $\E (\tau_\varepsilon)^r < \infty$ for all $\varepsilon >0$ and some $r>0$ implies
\[
\sum_{n=1}^\infty n^{r-1} \Ps(|\xi_n| > \varepsilon) < \infty,
\]
which defines the rate of convergence in the strong law.  If $\xi_n = (Y_1+\cdots+Y_n)/n$ and $Y_1,\dots,Y_n$ are i.i.d. zero-mean, then the necessary and sufficient condition for the $r-$quick convergence is the finiteness of the $(r+1)$th moment, $\E |Y_1|^{r+1} < \infty$. To study the first order asymptotic optimality
of the Shiryaev change-point detection rule in HMM, we will extend this idea to Markov chains.

Let $S_n^k$ denote the log-likelihood ratio between the hypotheses $\Hyp_k$ and $\Hyp_\infty$,
\begin{equation} \label{LLRgen}
S_n^{k} = \log \brc{\prod_{i=k}^n \Lambda_i }= \sum_{i=k}^n \log\frac{p_{0}(Y_{i}|\Yb_0^{i-1})}{p_{\infty}(Y_{i}|\Yb_0^{i-1})}, \quad k \le n.
\end{equation}
Assuming, for every $k>0$, the validity of a strong law of large numbers, i.e., convergence of $n^{-1} S_{k+n-1}^k$ to a constant $\Kc >0$ as $n\to\infty$, with a suitable rate, 
Tartakovsky and Veeravalli~\cite{TartakovskyVeerTVP05} 
proved that the Shiryaev procedure \eqref{Shirst1} with threshold $A=(1-\alpha)/\alpha$ is first-order asymptotically (as $\alpha\to 0$) optimal. Specifically,  they proved that the Shiryaev procedure minimizes asymptotically as $\alpha\to0$ in class $\classA$ the moments of the detection delay $\E[(T-\nu)^m | T\ge \nu]$ for $m\le r$  whenever $n^{-1} S_{k+n-1}^k$ converges to $\Kc$ $r-$quickly. 
Since this result is fundamental in the following study for HMMs, we now present an exact statement that summarizes the general asymptotic Bayesian theory. Recall that
$T_A$ denotes the Shiryaev change-point detection rule defined in \eqref{Shirst1}. It is easy to show that for an arbitrary general model, $\PFA(T_A) \le 1/(1+A)$ (cf. \cite{TartakovskyVeerTVP05}). Hence, selecting
$A=A_\alpha=(1-\alpha)/\alpha$ implies that $\PFA(T_{A_\alpha}) \le \alpha$ (i.e., $T_{A_\alpha}\in\classA$) for any $\alpha < 1-\omega_0$.
For $\varepsilon>0$, define the last entry time
\begin{equation}\label{LETLLRgen}
\tau_{\varepsilon, k} = \sup\set{n \ge 1: \abs{\frac{1}{n} S_{k+n-1}^k - \Kc} >  \varepsilon} \quad (\sup\{\varnothing\} = 0).
\end{equation}

\begin{theorem}[Tartakovsky and Veeravalli~\cite{TartakovskyVeerTVP05}]\label{Th:FOAOgen}
Let $r\ge 1$. Let the prior distribution of the change point satisfy condition \eqref{expon} and, in the case of $c=0$, let in addition $\sum_{k=0}^\infty |\log \omega_k|^{r} \omega_k<\infty$. 
Assume that $n^{-1}S_{k+n-1}^k$ converges $r-$quickly as $n\to\infty$ under $\Ps_k$ to some positive and finite number $\Kc$, 
i.e., $\E_k(\tau_{\varepsilon, k})^r<\infty$ for all $\varepsilon>0$ and all $k\ge 1$, and that 
\begin{equation}\label{rquickgen}
\sum_{k=0}^\infty \omega_k \E_k(\tau_{\varepsilon, k})^r  <\infty \quad \text{for all}~ \varepsilon >0 .
\end{equation}  

\noindent {\rm \bf (i)} Then for all $m \le r$
\begin{equation} \label{Momentsgen1}
\lim_{A\to\infty}  \frac{\E^\omega[(T_A-\nu)^m | T_A \ge \nu]}{(\log A)^m} =\frac{1}{(\Kc+c)^m}.
\end{equation}

\noindent {\rm \bf (ii)}   If  threshold $A=A_\alpha$ is selected so that $\PFA(T_{A_\alpha}) \le \alpha$ and $\log A_\alpha\sim |\log\alpha|$ as $\alpha\to0$, in particular $A_\alpha=(1-\alpha)/\alpha$, then the Shiryaev rule
is asymptotically optimal as $\alpha\to0$ in class $\classA$ with respect to moments of the detection delay up to order $r$, i.e., for all $0<m \le r$,
\begin{equation} \label{MADDAOgen}
\E^\omega[(T_{A_\alpha}-\nu)^m | T_{A_\alpha} \ge \nu] \sim  \brc{\frac{|\log\alpha|}{\Kc+c}}^m \sim \inf_{T \in \classA} \E^\omega[(T-\nu)^m | T\ge \nu]  \quad \text{as} ~\alpha \to 0.
\end{equation}
\end{theorem}

\begin{remark}
The assertions of Theorem~\ref{Th:FOAOgen} also hold true if the $r-$quick convergence condition \eqref{rquickgen} is replaced by the following two conditions: 
\begin{equation*}
\lim_{M\to\infty} \Ps_k\brc{\frac{1}{M} \max_{1 \le n \le M}S_{k+n-1}^k \ge (1+\varepsilon) \Kc} = 0 \quad \text{for all}~ \varepsilon >0~ \text{and all}~ k \ge 1
\end{equation*}  
and
\begin{equation*}
\sum_{k=0}^\infty \omega_k \E_k(\hat\tau_{\varepsilon, k})^r <\infty \quad \text{for all}~ \varepsilon >0 ,
\end{equation*}  
where $\hat\tau_{\varepsilon, k} = \sup\set{n \ge 1: n^{-1} S_{k+n-1}^k <   \Kc-\varepsilon}$ ($\sup\{\varnothing\} = 0)$. Condition  \eqref{rquickgen} guarantees both these conditions. 
\end{remark}

The first goal of the present paper is to specify these results for hidden Markov models. That is, we prove that the assertions of the above theorem hold for HMMs under some regularity conditions.
Moreover, by making use the specific structure of HMM and Markov nonlinear 
renewal theory, we also give the higher order asymptotic approximation to $\ADD(T)$. This is also our second goal.

\section{Problem Formulation for Hidden Markov Models}\label{s:FormulationHMM}

In this section, we define a finite state hidden 
Markov model as a Markov chain in a Markovian random environment, in which the underlying
environmental Markov chain can be viewed as latent variables. To be
more precise, let ${\bf X}= \{X_n, n \geq 0 \}$ be an ergodic (positive recurrent, irreducible and aperiodic) Markov chain on
a finite state space $D  = \{1, 2, \cdots, d\}$ with transition
probability matrix $ [p(x,x')]_{x,x'=1,\cdots,d}$ 
and stationary distribution $\pi=(\pi(x))_{x=1,\cdots,d}$. Suppose that a random sequence
$\{Y_n\}_{n=0}^{\infty},$ taking values in ${\bf R}^m$, is adjoined
to the chain such that $\{(X_n,Y_n), n \geq 0\}$ is a Markov chain
on $D \times {\bf R}^m$ satisfying $\Ps \{ X_1 \in A |
X_0=x,Y_0=y_0 \} = \Ps \{ X_1 \in A | X_0=x \}$ for $A \in 
{\mathscr B}(D)$, the Borel $\sigma$-algebra of $D$. Moreover, conditioning
on the full ${\bf X}$ sequence, we have
\begin{eqnarray} \label{ssm}
 \Ps \{Y_{n+1} \in  B | X_0,X_1,\ldots;Y_0,Y_1,\ldots, Y_n \}
= \Ps \{Y_{n+1} \in  B | X_{n+1}, Y_n \}
\end{eqnarray}
for each $n$ and $B \in {\mathscr B}({\bf R}^m),$ the Borel
$\sigma$-algebra of ${\bf R}^m$. Furthermore, 
let $f(\cdot|X_k, Y_{k-1})$ be the
transition probability density of $Y_k$ given $X_k$ and $Y_{k-1}$ with respect
to a $\sigma$-finite measure $Q$ on ${\bf R}^m$, such that
\begin{eqnarray}\label{ssmden}
 \Ps \{X_1 \in A, Y_{1} \in  B | X_0=x, Y_0=y  \}
= \sum_{x' \in A} \int_{y' \in B} p(x,x') f(y'|x',y)  Q(\d y'),
\end{eqnarray}
for $B \in{\mathscr B}({\bf R}^m)$.  We also assume that
the Markov chain $\{(X_n,Y_n), n \ge 0\}$ has a stationary
probability $\Gamma$ with probability density function $\pi(x)f(\cdot|x)$. 

Now we give a formal definition of the hidden Markov model.

\begin{definition}\label{def1}
The sequence $\{Y_n,n \ge 0\}$ is called a finite state hidden Markov model
(HMM) if there is an unobserved Markov chain $\{X_n,n \ge 0\}$
such that the process $\{(X_n,Y_n),n \ge 0\}$ satisfies {\rm
(\ref{ssm}) and (\ref{ssmden})}. 
\end{definition}

We are interested in the change-point detection problem for the HMM, which is of course a particular case of the general stochastic
model described in  \eqref{noniidmodel}. In other words, for 
$j =\infty, 0$, let $p_j(x,x')$ be the transition probability, 
$\pi_j(x)$ be the stationary probability, and $f_j(y'|x,y)$ be the 
transition probability density of the HMM in Definition~\ref{def1}. In the change-point problem, we suppose that the conditional density $f(y'|x,y)$
and the transition probability $p(x,x')$ change at an unknown time $\nu$ from $(p_\infty,f_\infty)$ to $(p_0,f_0)$.

Let $\Yb_0^n=(Y_0,Y_1,\ldots,Y_n)$ be the sample obtained from the HMM $\{Y_n, n \ge 0\}$ and denote
\begin{eqnarray}\label{LRn}
~~~LR_n  &:=& \frac{p_0(Y_0,Y_{1},\cdots,Y_n)}
{p_\infty(Y_0,Y_{1},\cdots,Y_n)}    \\
 &=& \frac{\sum_{x_0 \in D, \ldots, x_n \in D} \pi_0(x_0)f_0(Y_0|x_0)
   \prod_{l=1}^n  p_0(x_{l-1},x_l) f_0(Y_l|x_l,Y_{l-1})}
{\sum_{x_0 \in D, \ldots, x_n \in D} \pi_\infty(x_0)f_\infty(Y_0|x_0)
   \prod_{l=1}^n  p_\infty(x_{l-1},x_l) f_\infty(Y_l|x_l,Y_{l-1})} \nonumber
\end{eqnarray}
as the likelihood ratio. By \eqref{noniidmodel}, for $0 \le k \le n$,  the likelihood ratio of the hypothesis $\Hyp_k: \nu=k$ against $\Hyp_\infty: \nu=\infty$
for the sample $\Yb_0^n$ is given by  
\begin{equation} \label{LRkYn}
LR_n^k := \frac{p_k(\Yb_0^n)}{p_\infty(\Yb_0^n)} =  \frac{p_0(\Yb_k^n|\Yb_0^{k-1})}{p_\infty(\Yb_k^n|\Yb_0^{k-1})}  
= \prod_{i=k}^n\Lambda_i,
\end{equation}
where $\Lambda_i=p_0(Y_i|\Yb_0^{i-1})/p_\infty(Y_i|\Yb_0^{i-1})$.

Recall that in Section~\ref{s:Overview} we assumed that only the sample $\Yb_1^n=(Y_1,\dots,Y_n)$ can be observed and 
the initial value $Y_0$ is used for producing the observed sequence $\{Y_n, n =1,2,\dots\}$ with the desirable property. 
The initialization $Y_0$ affects the initial value of the likelihood ratio, $LR_0=\Lambda_0$, which can be either random or deterministic. 
In turn, this influences the behavior of $LR_n$ for $n \ge 1$.   Using the sample $\Yb_0^n$ in \eqref{LRn} and \eqref{LRkYn}  is convenient for Markov and hidden 
Markov models which can be initialized either randomly or deterministically.  If  $Y_0$ cannot be observed (or properly generated), then we assume $\Lambda_0=LR_0=1$, which is  equivalent to $f_0(Y_0|x)/f_\infty(Y_0|x)=1$ for all $x \in D$ in \eqref{LRn}. This is also the case when the change cannot occur before the observations become 
available, i.e., when $\omega_0=0$.

Of course, the probability measure (likelihood ratio) defined in (\ref{LRkYn}) is  one of several possible ways of representing the LR, when the change occurs at time $\nu=k.$ For instance,
when the post-change hidden state $X_{k}$ comes from
the pre-change hidden state $X_{k-1}$ with new transition probability, then the joint
marginal $\Ps_{k}$-distribution of $\Yb_0^n$  (with $n \ge k$) becomes
\begin{eqnarray}\label{LRnk}
LR_n^k  &:=& \frac{p_{0}(\Yb_0^n) / p_{0}(\Yb_0^{k-1})}{p_{\infty}(\Yb_0^n) /  p_{\infty}(\Yb_0^{k-1})} \nonumber  \\
  &\approx& \frac{\sum_{x_k \in D, \ldots, x_n \in D} \pi_0(x_k) f_0(Y_k|x_k) \prod_{l=k+1}^{n}   p_0(x_{l-1},x_l) f_0(Y_l|x_l,Y_{l-1})} 
  {\sum_{x_k \in D, \ldots, x_n \in D} \pi_\infty(x_k)f_\infty(Y_k|x_k)   
  \prod_{l=k+1}^{n}  p_\infty(x_{l-1},x_l) f_\infty(Y_l|x_l,Y_{l-1})}.
\end{eqnarray}
Note that the first equation in (\ref{LRnk}), the joint
marginal distributions formulation, is an alternative expression of (\ref{LRkYn}). 
In the second equation of (\ref{LRnk}), we approximate $p_j(x_{k-1},x_k) f_j(Y_k|x_k,Y_{k-1})$ by the associated stationary distribution $ \pi_j(x_k) f_j(Y_k|x_k)$ for  $j=\infty, 0$.

\begin{remark}
 As noted in Section~\ref{s:Intro}, in this paper, we consider the 
case of $\nu$ independent of the observed random variables. In the case that $\nu$ depends on the observed random variables,
a device of using pseudo-probability measures can be found in Fuh and Mei \cite{FuhMei15}.
\end{remark}

In the following sections, we 
investigate the Shiryaev change point detection rule defined in \eqref{Rnp} and \eqref{Shirst1}. We now give certain preliminary results required for this study.
Since the detection statistic 
$R_{n,\omega}$ involves $LR_n^k$ defined in (\ref{LRkYn}) and 
(\ref{LRnk}), we explore the structure of the likelihood ratio
$LR_n$ in (\ref{LRn}) first.
For this purpose, we represent (\ref{LRn}) as the ratio of
$L_1$-norms of products of Markovian random matrices. This
device has been proposed by Fuh \cite{FuhAS03} to study a recursive CUSUM change-point detection procedure in HMM. 
Here we carry out the same idea
to have a representation of the likelihood ratio $LR_n$.
Specifically, given a column vector $u=(u_1,\cdots,u_d)^t \in 
{\bf R}^d$, where $t$ denotes the transpose of the underlying 
vector in ${\bf R}^d$, define the $L_1$-norm
of $u$ as $\| u \|  = \sum_{i=1}^d |u_i|$.
The likelihood ratio $LR_n$ then can be represented as 
\begin{eqnarray}\label{proc}
LR_n= \frac{p_0(Y_0, Y_1,\cdots,Y_n)}{p_\infty(Y_0,Y_1,\cdots,Y_n)}
=\frac{ \| M^0_n \cdots M^0_1 M^0_0 \pi_0\|}
{ \| M^\infty_n \cdots M^\infty_1 M^\infty_0 \pi_\infty\|},
\end{eqnarray}
where 
\begin{equation}\label{rm1}
M^j_0 = \left[ \begin{array}{cccc}
f_j(Y_0|X_0=1) & 0 & \cdots & 0 \\
\vdots  &\ddots & \vdots  & \vdots\\
0 & 0 & \cdots & f_j(Y_0|X_0=d) \end{array} \right], 
\end{equation}
\begin{equation}\label{rm2}
M^j_k= \left[ \begin{array}{ccc}
p_j(1,1)f_j(Y_k|X_{k}=1,Y_{k-1}) & \cdots &
p_j(d,1)f_j(Y_k|X_{k}=1,Y_{k-1})\\
\vdots & \ddots  & \vdots\\
p_j(1,d)f_j(Y_k|X_{k}=d,Y_{k-1}) & \cdots &
p_j(d,d)f_j(Y_k|X_{k}=d,Y_{k-1}) \end{array} \right]
\end{equation}
for $j=0,\infty$, $k=1,\cdots,n$, and
\begin{eqnarray}
 \pi_j = \big(\pi_j(1),\cdots,\pi_j(d) \big)^t.
\end{eqnarray}

Note that each component $p_j(x,x')f_j(Y_k|X_{k}=x',Y_{k-1})$ 
in $M_k^j$ represents $X_{k-1} =x$ and $X_k=x'$, and $Y_k$ is a 
random variable with transition probability density $f_j(Y_k|x',Y_{k-1})$
for $j=0,\infty$ and $k=1,\cdots,n$. Therefore the $M_k^j$ are random
matrices. Since $\{(X_n,Y_n), n \ge 0 \}$ is a Markov chain
by definition (\ref{ssm}) and (\ref{ssmden}), this implies that 
$\{M_k^j,k =1,\cdots,n\}$ is a sequence of Markov random matrices 
for $j=0,\infty.$ 
Hence, $LR_n$ is the ratio of the $L_1$-norm of the products of Markov random matrices via 
representation (\ref{proc}). Note that $\pi_j(\cdot)$ is fixed in (\ref{proc}).

Let $\{(X_n,Y_n), n \ge 0 \}$ be the Markov chain defined in (\ref{ssm}) and (\ref{ssmden}).
Denote $Z_n :=(X_n,Y_n)$ and $\bar{D}:=D \times {\bf R}^m$. 
Define $Gl(d,{\bf R})$ as the set of invertible $d \times d$ matrices with real entries.  
For given $k=0,1,\cdots,n$ and $j=\infty, 0$, let $M_k^j$ be the random matrix from
$\bar{D} \times \bar{D}$ to $Gl(d,{\bf R})$, as defined in (\ref{rm1}) and (\ref{rm2}). 
For each $n$, let
\begin{eqnarray} \label{rfs}
 {\bf M}_n =  ({\bf M}_n^\infty, {\bf M}_n^0) 
= \big(M^\infty_n \circ \cdots \circ M^\infty_1\circ M^\infty_0,
M^0_n \circ \cdots \circ M^0_1 \circ M^0_0\big) .
\end{eqnarray}
Then the system $\{(Z_n,{\bf M}_n), n \ge 0 \}$ is
called a product of Markov random matrices on $\bar{D} \times Gl(d,{\bf R})\times Gl(d,{\bf R})$.
Denote ${\Ps}^z$ as the probability distribution of $\{(Z_n,{\bf M}_n), n \ge 0 \}$ 
with $Z_0 =z$, and ${\E}^z$ as the expectation under ${\Ps}^z$.

Let $u \in {\bf R}^d$ be a $d$-dimensional vector, $\overline{u}:=u/||u||$ the normalization of 
$u$ ($||u|| \neq 0$), and denote $P({\bf R}^d)$ as the projection space of ${\bf R}^d$
which contains all elements $\overline{u}$.
For given $\overline{u} \in P({\bf R}^d)$ and $M \in Gl(d,{\bf R})$, denote
$M \cdot \overline{u} = \overline{Mu}$ and $\overline{{\bf M}_k u} = 
\big(\overline{{\bf M}_k^\infty u}, \overline{{\bf M}_k^0 u}\big)$, for $k=0,\cdots,n$. Let
\begin{eqnarray}\label{tk}
W_0 = ( Z_0, \overline{{\bf M}_0u}),~W_1 =(Z_1, \overline{{\bf M}_1 u}), \cdots,
W_n = ( Z_n ,\overline{{\bf M}_n u}).
\end{eqnarray}
Then, $\{W_n, n \ge 0\}$ is a Markov chain on the state space
$\bar{D} \times P({\bf R}^d)\times P({\bf R}^d)$ with the transition kernel
\begin{eqnarray}\label{tk1}
{\Ps}((z,\overline{u}),A \times B):={\E}^z (I_{A \times B} (Z_1,\overline{M_1 u}))
\end{eqnarray}
for all $z \in \bar{D},~\overline{u}:=(\overline{u}, \overline{u})\in P({\bf R}^d) 
\times P({\bf R}^d),~A \in {\cal B}(\bar{D})$,
and $B \in {\cal B}(P({\bf R}^d)\times P({\bf R}^d))$, the $\sigma$-algebra of 
$P({\bf R}^d)\times P({\bf R}^d)$. For simplicity, we let 
${\Ps}^{(z,\overline{u})}:={\Ps}(\cdot,\cdot)$ and
denote ${\E}^{(z,\overline{u})}$ as the expectation under ${\Ps}^{(z,\overline{u})}$.
Since the Markov chain $\{(X_n,Y_n), n \ge 0\}$ has transition
probability density and the random matrix $M_1(\theta)$ is driven by $\{(X_n,Y_n), n \ge 0\}$, it 
implies that the induced transition probability
$\Ps(\cdot,\cdot)$ has a density with respect to $m  \times Q$.
Denote the density as $p$ for simplicity. 
According to Theorem 1(iii) in Fuh \cite{FuhAS03}, under conditions {\bf C1} and {\bf C2} 
given below, the stationary distribution of $\{W_n, n\ge 0\}$ 
exists. Denote it by $\Pi$.

The crucial observation is that the log-likelihood ratio can now be written as an additive
functional of the Markov chain $\{W_n, n \ge 0\}$. That is,
\begin{eqnarray}\label{lrn}
S_n=\log LR_n =\sum_{k=1}^{n}g(W_{k-1},W_k),
\end{eqnarray}
where
\begin{eqnarray}\label{gk}
~~~ g(W_{k-1},W_k)
:= \log \frac{||M^0_k \circ \cdots \circ M^0_1  M^0_0 \pi_0||/
||M^0_{k-1} \circ \cdots \circ M^0_1 M^0_0 \pi_0||}
{||M^\infty_k \circ \cdots \circ M^\infty_1 M^\infty_0 \pi_\infty||/
||M^\infty_{k-1} \circ \cdots \circ M^\infty_1 M^\infty_0 \pi_\infty||}.
\end{eqnarray}

In the following sections, we show that the Shiryaev procedure with a certain
threshold $A=A_\alpha$ is asymptotically first-order optimal as
$\alpha \to 0$ for a large class of prior distributions and provide
a higher order approximation to the average detection delay 
for the geometric prior.

Regarding prior distributions $\omega_k=\Ps(\nu=k)$, we will assume throughout that condition 
\eqref{expon} holds for some $c \ge 0$.
A case where a fixed positive $c$ is replaced with the value of $c_\alpha$ that depends on $\alpha$ and vanishes when $\alpha \to 0$ with a certain appropriate rate will also be handled.

\section{First Order Asymptotic Optimality}\label{s:FOAO}

For ease of notation, let ${\cal X} := \bar{D} \times P({\bf R}^d)\times P({\bf R}^d)$ be the 
state space of the Markov chain $\{W_n, n \ge 0\}$. Denote $w:=(z, \bar{u}, \bar{u})$ 
and $\bar{w} := (z_0, \pi_\infty, \pi_0)$, where $z_0$ is the initial state of 
$Z_0$ taken from $\pi_\infty f_\infty(\cdot|x)$. To prove first-order asymptotic 
optimality of the Shiryaev rule and to derive a high-order asymptotic approximation to the average 
detection delay for HMMs, the conditions C1--C2 set below are assumed
throughout this paper. Before that we need the following definitions and notations.

Abusing the notation a little bit, a Markov chain $\{X_n,n \ge 0\}$ on a general state space 
${\cal X}$  is called $V$-uniformly ergodic if there exists a measurable function
$V: {\cal X} \rightarrow [1,\infty)$, with $\int V(x) m(dx) < \infty$, such that 
\begin{eqnarray}\label{3.11}
~~~ \lim_{n \rightarrow \infty} \sup_{x \in {\cal X}} \bigg\{\frac{\big|\E[h(X_n)|X_0=x]
 - \int h(y)m(dy)\big|}{V(x)}: |h| \le V \bigg\} =0.
\end{eqnarray}
Under irreducibility and aperiodicity assumption, $V$-uniform   
ergodicity implies that $\{X_n,n \ge 0\}$ is Harris recurrent in
the sense that there exist a recurrent set ${\cal A} \in \Bmc({\cal X})$, a probability measure $\varphi$ on ${\cal A}$ and an
integer $n_0$ such that $\Ps\{ X_n \in {\cal A}~{\rm for~some}~n \ge
n_0 |X_0=x\} =1$ for all $x \in {\cal X}$, and there exists
$\lambda > 0$ such that
\begin{eqnarray}\label{min}
\Ps\{ X_n \in A |X_0=x\} \ge \lambda \varphi(A)
\end{eqnarray}
for all $x \in {\cal A}$ and $A \subset {\cal A}$. Under (\ref{min}), Athreya and Ney 
\cite{AthreyaNeyTrans78} proved that $X_n$ admits a regenerative scheme with i.i.d. 
inter-regeneration times for an augmented Markov chain, which is called the ``split
chain''. Recall that $S_n= \log LR_n$ is defined in (\ref{lrn}). 
Let $\varrho$ be the first time $(>0)$ to reach the
atom of the split chain, and define $u(\alpha,\zeta)= \E^{\mu}e^{\alpha S_\varrho-\zeta \varrho}$ 
for $\zeta \in {\bf R}$, where $\mu$ is an initial distribution on ${\cal X}$.
Assume that
\begin{eqnarray}\label{mom}
 \{(\alpha,\zeta):u(\alpha,\zeta) < \infty\}~~~{\rm is~an~open~subset~on}~{\bf R}^2.
\end{eqnarray}
Ney and Nummelin \cite{NeyNummelinAP87} showed that
${\cal D}=\{\alpha:u(\alpha,\zeta) < \infty~{\rm for~some}~\zeta\}$ is an
open set and that for $\alpha \in {\cal D}$ the transition kernel
$\hat{\Ps}_{\alpha}(x,A ) = \E^{x}\{ e^{\alpha S_1} I_{\{X_1 \in A
\}}\}$ has a maximal simple real eigenvalue $e^{\Psi(\alpha)}$,
where $\Psi(\alpha)$ is the unique solution of the equation
$u(\alpha,\Psi(\alpha))=1$ with the corresponding eigenfunction
$r^*(x;\alpha):= \E^{x} \exp\{\alpha S_\varrho - \Psi(\alpha)
\varrho\}$. Here $\E^x(\cdot)=\E(\cdot |X_0=x)$. For a measurable subset $A \in \Bmc({\cal X})$ and
$x \in {\cal X}$, define
\begin{eqnarray}\label{eignmea}
~~~~~~~{\cal L}(A ;\alpha) =  \E^{\mu}\bigg[\sum_{n=0}^{\varrho -1} e^{\alpha S_n- n \Psi(\alpha)}
I_{\{X_n \in A \}}\bigg],~~~
{\cal L}_{x}(A;\alpha) =  \E^{x}\bigg[\sum_{n=0}^{\varrho -1} e^{\alpha S_n- n \Psi(\alpha)}
I_{\{X_n \in A \}}\bigg].
\end{eqnarray}
Note that the finiteness of the state space ensures the finiteness of the eigenfunction $r(x;\alpha)$ and the eigenmeasure ${\cal L}(A ;\alpha)$.

For given ${\Ps}_{\infty}$ and ${\Ps}_{0}$, define the Kullback--Leibler information number as
\begin{eqnarray}\label{kl}
 \Kc:=K({\Ps}_{0}, {\Ps}_{\infty}) = {\E}_{{\Ps}_{0}}^{\Pi} \bigg( \log \frac{\|M_1^0  M_0^0 \pi_{0} \| }
 {\|M_1^\infty M_0^\infty \pi_{\infty} \|}  \bigg), 
\end{eqnarray}
where ${\Ps}_{\infty}$ (${\Ps}_{0}$) denotes the probability
of the Markov chain $\{W_n^{\infty}, n \ge 0\}$ ($\{W_n^{0}, n \ge
0\}$), and ${\E}_{{\Ps}_{\infty}}^{\Pi}$ (${\E}_{{\Ps}_{0}}^{\Pi}$) refers to the expectation 
for ${\Ps}_{\infty}$ (${\Ps}_{0}$) under the invariant probability $\Pi$.

The following conditions {\bf C} are assumed throughout this paper.

{\bf C1.} For each $j=\infty,0$, the Markov chain $\{X_n, n \ge 0 \}$ 
defined in (\ref{ssm}) and (\ref{ssmden}) is ergodic (positive recurrent, irreducible and aperiodic) 
on a finite state space $D=\{1,\cdots,d\}$. Moreover,
the Markov chain $\{(X_n,Y_n), n \ge 0\}$ is irreducible, aperiodic and
$V$-uniformly ergodic for some $V$ on $\bar{D}$ with 
$$
 \sup_{(x,y)\in D \times {\bf R}^m} \set{\frac{\E^{(x,y)}[V((X_1,Y_1))]}{V((x,y))}} < \infty.
 $$
We also assume that the Markov chain
$\{(X_n,Y_n),n \ge 0\}$ has stationary probability $\Gamma$ with probability density 
$\pi_xf(\cdot|x)$ with respect to a $\sigma$-finite measure.

{\bf C2.} Assume that $0< \Kc < \infty$. For each $j=\infty,0$, assume that the random matrices 
$M_0^j$ and $M_1^j$, defined in (\ref{rm1}) and (\ref{rm2}), are invertible 
$\Ps_j$ almost surely and for some $r \ge 1$
\begin{eqnarray} \label{C22}
 \sup_{(x, y) \in D \times {\bf R}^m} \E^{(x,y)} \bigg| \frac{|Y_1|^r V((X_1,Y_1))}{V((x,y))}
 \bigg|  < \infty.
\end{eqnarray}

\begin{remark} \label{Rem1}
The ergodicity condition C1 for Markov chains is quite general and covers several interesting examples of HMMs in Section~\ref{s:Examples}. 
Condition C2 is a standard constraint imposed on the Kullback--Leibler information number and a moment condition under the variation norm. 
Note that positiveness of the Kullback--Leibler information number is not at all restrictive since it holds whenever the probability density functions of
${\Ps}_{0}$ and ${\Ps}_{\infty}$ do not coincide almost everywhere. The finiteness condition is quite natural and holds in most cases. 
Moreover, the cases where it is infinite are easy to handle and can be viewed as degenerate from the asymptotic theory standpoint. 
\end{remark}

Recall that the Markov chain $\{W_n, n \ge 0\}$ on ${\cal X}:= \bar{D} \times P({\bf R}^d)\times P({\bf R}^d)$ is induced by the products of random matrices.
A positivity hypothesis of the elements of $M_k^j$, which implies the positivity on the functions in the support of the Markov chain, leads to contraction properties which are the basis of the 
spectral theory developed by Fuh \cite{FuhAS03}. Another natural assumption is that the transition
probability possesses a density. This leads to a classical situation in the context of the 
so-called ``Doeblin condition'' for Markov chains. It also leads to precise results of 
the limiting theory and has been used to develop a nonlinear renewal theory in Section 3 of Fuh 
\cite{FuhAS04}. We summarize the properties of $\{W_n,n \ge 0\}$ in the following proposition. 
Since the proof is the same as Proposition~2 of Fuh \cite{FuhAS04}, it is omitted.
Denote $ \chi(M) = \sup ( \log\|M\|, \log\|M^{-1}\|)$.

\begin{proposition}\label{prop1}
Consider a given HMM as in \eqref{ssm} and \eqref{ssmden} satisfying {\rm C1--C2}. Then the induced Markov chain $\{W_n,n \ge 0\}$, 
defined in \eqref{tk} and \eqref{tk1}, is an aperiodic, irreducible and 
Harris recurrent Markov chain under $\Ps_\infty$. Moreover, it is also a
$V$-uniformly ergodic Markov chain for some $V$ on ${\cal X}$. We have
$\sup_{w} \{\E_\infty [V(W_1)|W_0=w]/V(w)\} < \infty,$ and there exist $a,C > 0$ such that 
${\E}_{\infty}(\exp\{a \chi(M_1)\}|W_0=w) \le C$ for all $w=(x_0,\pi,\pi) \in {\cal X}.$
\end{proposition}

Recall (see \eqref{LLRgen}) that by $S_n^k = \log  LR_n^k$, $k \le n$ ($S_0^k=0$, $S_n^{n+j} =0$),
we denote the log-likelihood ratio (LLR) of the hypotheses that the change 
takes place at $\nu=k $ and there is never a change ($\nu=\infty$).
By Theorem~\ref{Th:FOAOgen}(ii), the $r-$quick convergence of the normalized LLR processes $n^{-1}S_{k+n-1}^k$, $k=1,2,\dots,$ to the Kullback--Leibler information number $\Kc$ 
(see \eqref{rquickgen}) is sufficient for asymptotic optimality of the Shiryaev procedure. Thus, to establish asymptotic optimality for HMMs, it suffices to show that
condition \eqref{rquickgen} is satisfied under conditions C1--C2. This is the subject of the next lemma, which is then used for proving asymptotic optimality of Shiryaev's procedure for HMMs.

Let $\tau_{\varepsilon, k}$ be the last entree time defined in \eqref{LETLLRgen}.
Denote $\Ps^{(\pi,f)}_k$ as the probability of the Markov chain $\{(X_n,Y_n), n \ge 0\}$ 
starting with initial distribution $(\pi,f)$, the stationary distribution, and conditioned on the change point $\nu=k$. Let $\E^{(\pi,f)}_k$ be the corresponding expectation. For notational simplicity, we omit $(\pi,f)$ and simply denote $\Ps_{k}$ and $\E_{k}$ from now on. 

\begin{lemma} \label{lem1} 
Assume that conditions {\rm C1--C2} hold.  

{\rm \bf (i)} As $n\to\infty$,
\begin{equation}\label{as}
\frac{1}{n}S_{k+n-1}^{k} \longrightarrow  \Kc \quad  \Ps_{k}-\text{a.s. for every} ~ k<\infty.
\end{equation}

{\rm \bf (ii)} Let $\E_{1}|S_1^1|^{r+1} < \infty$ for some $r\ge 1$. Then 
$\E_{k} (\tau_{\varepsilon, k})^r < \infty$ 
for all $\varepsilon > 0$ and $k \ge 1$, i.e., as $n\to\infty$,
\begin{equation}\label{cas}
\frac{1}{n}S_{k+n-1}^{k} \longrightarrow  \Kc \quad \Ps_{k}-r-\text{quickly for every}~k<\infty.
\end{equation}
\end{lemma}

\begin{remark}
Since the Markov chain $\{W_n, n \ge 0\}$ is stationary if the initial distribution is $(\pi, f)$,
Lemma \ref{lem1}(ii) implies that a stronger uniform version holds, i.e., 
$\sup_{k\ge 1}  \E_{k} (\tau_{\varepsilon, k})^r < \infty$ for all $\varepsilon > 0$. 
This implies that 
\begin{equation}\label{cas1}
\sum_{k=0}^\infty \omega_k \E_{k} (\tau_{\varepsilon, k})^r<\infty \quad \text{for all $\vae>0$}.
\end{equation}
\end{remark}

\proof

(i) By Proposition \ref{prop1}, Proposition 1 in Fuh \cite{FuhAS03}, and the ergodic theorem for Markov chains, it is easy to see that  the upper Lyapunov exponent for the Markov chain $\{W_n, n\ge 0\}$ 
under the probability $\Ps_0$ is nothing but the relative entropy defined as 
\[
\begin{aligned}
H(\Ps_0,\Ps_j) &= {\E}_{\Ps_0}\brcs{\log p_j(Y_1|Y_0, Y_{-1},\cdots)} \\
  &= {\E}_{\Ps_0}\brcs{\log \sum_{x=1}^d \sum_{x'=1}^d f_j(Y_1|Y_0;X_1=x')   p_j(x,x') \Ps_j (X_0=x|Y_0, Y_{-1},\cdots)}\\
&= {\E}_{\Ps_0}\brcs{\log \|M_1^j \Ps_j (X_0=x|Y_0, Y_{-1},\cdots)\|}. 
\end{aligned}
\]
Therefore, the Kullback--Leibler information number can be defined as that in \eqref{kl}. Hence (i) is proved via the standard law of large numbers 
argument for Markov random walks.

(ii) Note that
\[
\Ps_{k} (\tau_{\vae,k} \ge n) \le  \Ps_{k}\left \{\sup_{j \ge n} \left | \frac{S_{k+j-1}^k}{j} - \Kc  \right |>\vae \right\},
\]
so that
\[
\E_{k}(\tau_{\vae,k})^r \le r \int_0^\infty t^{r-1} \Ps_{k}\left\{\sup_{j \ge t} \left| \frac{S_{k+j}^k}{j} - \Kc \right|>\vae
\right\} \, \d t,
\]
and, therefore, in order to prove (ii) it suffices to prove that
\begin{equation}\label{casra}
\sum_{k=1}^\infty \omega_k \sum_{n=1}^\infty n^{r-1} \Ps_k\set{\sup_{j \ge n} \Bigl| \frac{S_{k+j-1}^k}{j} - \Kc \Bigr|>\vae}<\infty \quad \text{for all $\vae>0$}.
\end{equation}

By Proposition~\ref{prop1} above, we can apply Theorem 6 of
Fuh and Zhang \cite{FuhZhang00} to obtain that \eqref{casra} holds whenever 
$\E_{1} | S_{1}^1|^2 < \infty$ and $\E_{1} [(S_{1}^1 - \Kc)^+]^{r+1}<\infty$ for some $ r \ge 1$, 
and the following Poisson equation has solution $\Delta : {\cal X} \rightarrow {\bf R}$
\begin{eqnarray}\label{6.6a}
\E^{w} \Delta(W_{1}) - \Delta(w) = \E^{w} S_{1}^1 - \E^\Pi S_1^1
\end{eqnarray}
for almost every $w \in {\cal X}$ with $\E^\Pi \Delta (W_{1}) = 0$, where $\E^w(\cdot)=
\E_1(\cdot|W_0=w)$ is the conditional expectation when the change occurs at $\nu=1$ 
conditioned on $W_0=w$, i.e., when the Markov chain $\{W_n\}_{n\ge 0}$ is initialized from the point $w$, and $\E^{\Pi}(\cdot) = \int \E^w(\cdot) \, \d \Pi $.

To check the validity of (\ref{6.6a}), we first note that using Proposition \ref{prop1} and Theorem 17.4.2 of Meyn and Tweedie \cite{MeynTweedie09}, we have the existence of a solution for (\ref{6.6a}). 
Moreover, $\sup_{j\ge 0} \E^{w} |\Delta(W_{j})|^r < \infty$ for some $r \ge 1$ follows from the boundedness property in Theorem 17.4.2 of Meyn and Tweedie \cite{MeynTweedie09} 
under conditions C1--C2.
Next, by C2 and the moment assumption $\E_{1}|S_1^1|^{r+1} < \infty$ in (ii), 
the moment conditions $\E_{1} | S_{1}^1|^2 < \infty$ and $\E_{1} [(S_{1}^1 - \Kc)^+]^{r+1}<\infty$ hold for some $ r \ge 1$. This completes the proof.
\endproof

\begin{remark}
The assertions of Lemma~\ref{lem1} hold true even if the Markov chain $W_n$ is initialized from any deterministic or random point with ``nice'' distribution. 
However, the proof in this case becomes more complicated.
\end{remark}

Now everything is prepared to prove the first-order optimality property of the Shiryaev changepoint detection procedure.

\begin{theorem} \label{Th:th1} 
Let conditions  {\rm C1--C2} hold. Furthermore, let $r\ge 1$ and assume $\E_1|S_1^1|^{r+1} < \infty.$   
Let the prior distribution of the change point satisfy condition \eqref{expon} and, in the case of $c=0$, let in addition 
$\sum_{k=0}^\infty |\log \omega_k|^{r} \omega_k<\infty$. 

{\rm \bf (i)}  Then for all $m \le r,$ 
\begin{equation}\label{Momentsgen}
\E^\omega[(T_A-\nu)^m | T_A \ge \nu] \sim  \brc{\frac{\log A}{\Kc + c}}^m \quad \text{as} ~ A \to \infty.
\end{equation}

{\rm \bf (ii)} Let $A=A_\alpha=(1-\alpha)/\alpha$.
Then $\pfa(T_{A_\alpha} )\le \alpha$, i.e., the Shiryaev detection
procedure $T_{A_\alpha}$ belongs to  class $\classA$, and
\begin{equation}\label{MomentsAOgen}
\inf_{T \in \classA} \E^\omega[(T - \nu)^m | T \ge \nu]  \sim \brc{\frac{|\log \alpha|}{\Kc + c}}^m \sim
\E^\omega[(T_{A_\alpha}-\nu)^m | T_{A_\alpha} \ge \nu]   \quad \text{as}~ \alpha\to 0.
\end{equation}
This assertion also holds if threshold $A_\alpha$ is selected so that $\PFA(T_{A_\alpha})\le \alpha$ and $-\log\PFA(T_{A_\alpha})\sim \log A_\alpha$ as $\alpha\to0$.
\end{theorem}

\proof
Part (i) follows from Theorem~\ref{Th:FOAOgen}(ii) and (\ref{cas})--\eqref{cas1}. 

(ii) It is straightforward to show that  $\pfa(T_{A_\alpha}) \le \alpha$ when threshold $A_\alpha=(1-\alpha)/\alpha$
(see, e.g., Tartakovsky and Veeravalli \cite{TartakovskyVeerTVP05}). 
Asymptotic formulas \eqref{MomentsAOgen}
follow from Theorem~\ref{Th:FOAOgen}(ii) and (\ref{cas})--\eqref{cas1}.
\endproof

Theorem~\ref{Th:th1} covers a large class of prior distributions both with exponential tails and heavy tails. However, condition \eqref{expon} does not include the case with positive exponent 
which can vanish ($c\to 0$). 
Indeed, in this case, the sum $\sum_{k=0}^\infty |\log \omega_k|^{r} \omega_k$ becomes infinitely large and the results of the theorem are applicable only under an additional restriction 
on the rate with which this sum goes to infinity. The following theorem addresses this issue when $\omega_k=\omega_k^\alpha$ depends on the $\PFA(T_{A_\alpha})$ and $c=c_\alpha\to0$ as $\alpha\to0$.  

\begin{theorem}\label{Th:th2}
Let  the prior distribution $\{\omega_k^\alpha\}_{k\ge 0}$ satisfy condition \eqref{expon} with 
$c=c_\alpha\to 0$ as $\alpha\to 0$ in such a way that 
\begin{equation}\label{Prior3hmm}
\lim_{\alpha\to 0} \frac{{\sum_{k=0}^\infty |\log \omega_k^\alpha|^r \omega_k^\alpha}}{|\log \alpha|^r} = 0.
\end{equation} 
Assume that conditions {\rm C1--C2} are satisfied and that, in addition, $\E_1|S_1^1|^{r+1} < \infty$ for some $r\ge 1$.  If $A=A_\alpha$ is so selected that $\PFA(T_{A_\alpha}) \le \alpha$ and 
$\log A_\alpha\sim |\log\alpha|$ as $\alpha\to0$, in particular $A_\alpha=(1-\alpha)/\alpha$, then the Shiryaev procedure $T_{A_\alpha}$ is asymptotically optimal as $\alpha\to0$ in class $\classA$, 
minimizing moments of the detection delay up to order $r$: for all $0<m \le r$
\begin{equation} \label{MADDAOhmm}
 \E^\omega[(T_{A_\alpha}-\nu)^m | T_{A _\alpha}\ge \nu] \sim  \brc{\frac{|\log\alpha|}{\Kc}}^m \sim \inf_{T \in \classA} \E^\omega[(T-\nu)^m | T \ge  \nu]  \quad \text{as}~ \alpha \to 0.
\end{equation}
The assertion also holds true for heavy-tailed prior distributions, i.e., when $c=0$ in \eqref{expon}, if condition \eqref{Prior3hmm} is satisfied. 
\end{theorem}

\begin{proof}
We first establish an asymptotic lower bound for moments of the detection delay in class $\classA$. 
For $0<\varepsilon <1$ and small positive $\delta_\alpha$, 
let $N_{\alpha,\varepsilon}= (1-\varepsilon) |\log\alpha|/(\Kc+c_\alpha+\delta_\alpha)$. 
By the Markov inequality, for any stopping time $T$,
\[
\ADD(T) \ge \E^\omega(T-\nu)^+  \ge N_{\alpha,\varepsilon}\brcs{\Ps^\omega(T-\nu \ge 0) 
- \Ps^\omega(0 \le T-\nu < N_{\alpha,\varepsilon})}.
\]
Since for any $T\in\classA$, $\Ps^\omega(T \ge \nu) \ge 1-\alpha$, we have 
\begin{equation} \label{ADDlower}
\inf_{T\in\classA}\ADD(T) \ge N_{\alpha,\varepsilon}\brcs{1- \alpha - \sup_{T\in \classA}
\Ps^\omega(0 \le T-\nu < N_{\alpha,\varepsilon})}.
\end{equation}
It follows from Lemma~\ref{lem1}(i) that $n^{-1}S_{k+n-1}^{k}$ converges to $\Kc$ almost surely under $\Ps_k$, and therefore, for all $\varepsilon >0$ and $k\ge 1$
\begin{equation}\label{Pmax2}
\Ps_k\brc{\frac{1}{M} \max_{1 \le n \le M}S_{k+n-1}^k \ge (1+\varepsilon) \Kc} \to 0 \quad \text{as}~ M\to\infty.
\end{equation}
So we can apply Lemma~\ref{Lem:A1} in the Appendix (see Section \ref{s:Proofs}), according to which for any 
$\varepsilon \in (0,1)$,
\[
\sup_{T\in \classA}\Ps^\omega(0 \le  T-\nu < N_{\alpha,\varepsilon}) \to 0 \quad \text{as}~ \alpha\to 0.
\]
This along with \eqref{ADDlower} yields
\[
\inf_{T\in\classA} \ADD(T) \ge (1-\varepsilon) \frac{|\log\alpha|}{\Kc} (1+o(1)) \quad \text{as}~ \alpha\to0.
\]
Now, by Jensen's inequality, for any $m\ge 1$, $\E^\omega[(T-\nu)^m | T \ge\nu] \ge [\ADD(T)]^m$, which yields 
\[
\inf_{T\in\classA} \E^\omega[(T-\nu)^m | T\ge\nu] \ge (1-\varepsilon)^m \brc{\frac{|\log\alpha|}{\Kc}}^m (1+o(1)).
\]
Since $\varepsilon$ can be arbitrarily small, we obtain the asymptotic lower bound 
\begin{equation}\label{Lowergendalpha}
\inf_{T\in\classA} \E^\omega[(T-\nu)^m | T \ge \nu] \ge \brc{\frac{|\log\alpha|}{\Kc}}^m (1+o(1)) \quad \text{as}~ \alpha\to0,
\end{equation}
which holds for all $m>0$.

We now show that under conditions C1--C2 and $\E_1|S_1^1|^{r+1} < \infty$, the following asymptotic upper bound holds:
\begin{equation}\label{UpperTA}
\E^\omega[(T_{A_\alpha}-\nu)^r | T_{A_\alpha} \ge\nu] \le \brc{\frac{|\log\alpha|}{\Kc}}^r (1+o(1))  \quad \text{as}~ \alpha\to0
\end{equation}
as long as $\log A_\alpha\sim |\log\alpha|$. Obviously, this upper bound together with the previous lower bound proves the assertion of the theorem.

Evidently, for any $A>0$ and $k\ge 0$,
\[
(T_A-k)^+ \le \eta_A(k) = \inf\set{n \ge 1: \twS_{k+n-1}^k \ge \log(A/\omega_k^\alpha)},
\]
where $\twS_{k+n-1}^k = S_{k+n-1}^k -\log \Ps(\nu> n+k-1)$. Thus, 
\[
\E^\omega[(T_{A_\alpha}-\nu)^m | T_{A_\alpha} \ge \nu] \le \frac{\sum_{k=0}^\infty \omega_k^\alpha \E_k[\eta_{A_\alpha}(k)]^m}{1-\alpha}
\]
and to establish inequality \eqref{UpperTA} it suffices to show that 
\begin{equation}\label{Uppereta}
\sum_{k=0}^\infty \omega_k^\alpha \E_k[\eta_{A_\alpha}(k)]^r \le \brc{\frac{|\log\alpha|}{\Kc}}^r (1+o(1))  \quad \text{as}~ \alpha\to0.
\end{equation}
Define the last entry time 
\[
\tilde{\tau}_{\varepsilon,k} = \sup \set{n \ge 1: n^{-1} \twS_{k+n-1}^k - \Kc 
-c_\alpha < - \varepsilon} \quad (\sup\{\varnothing\} = 0) .
\]
It is easy to see that 
\[
(\eta_{A_\alpha}(k)-1)(\Kc+c_\alpha-\varepsilon) \le \twS_{k+\eta_{A_\alpha}(k)-2}^k <  \log(A_\alpha/\omega_k^\alpha) \quad \text{on}~ \set{\tilde{\tau}_{\varepsilon,k} +1 < \eta_{A_\alpha}(k) < \infty }.
\]
It follows that for every $0 < \varepsilon < \Kc+c_\alpha$
\[
\begin{aligned}
\brcs{\eta_{A_\alpha}(k)}^r  & \le  \brcs{1+ \frac{\log (A_\alpha/\omega_k^\alpha)}{\Kc+c_\alpha -\varepsilon} \Ind{\eta_{A_\alpha}(k)
> \tilde{\tau}_{\varepsilon,k}+1} + (\tilde{\tau}_{\varepsilon,k}+1) \Ind{\eta_{A_\alpha}(k) \le \tilde{\tau}_{\varepsilon,k}}}^r
\\
& \le  \brc{\frac{\log (A_\alpha/\omega_k^\alpha)}{\Kc+c_\alpha-\varepsilon} +\tilde{\tau}_{\varepsilon,k}+2}^r.
\end{aligned}
\]

Since $-n^{-1} \log \Ps(\nu>k+n-1) \to c_\alpha$ as $n \to \infty$, by \eqref{cas}--\eqref{cas1}
\[
\sum_{k=0}^\infty \omega_k^\alpha \E_k(\tilde\tau_{\varepsilon, k})^r  <\infty \quad \text{for all}~ \varepsilon >0, 
\]
and using condition \eqref{Prior3hmm} and the fact that $\log A_\alpha\sim |\log\alpha|$, we finally obtain that for any $0<\varepsilon < \Kc$
\[
\sum_{k=0}^\infty \omega_k^\alpha \E_k[\eta_{A_\alpha}(k)]^r \le \brc{\frac{|\log\alpha|}{\Kc-\varepsilon}}^r (1+o(1))  \quad \text{as}~ \alpha\to 0.
\]
Since $\varepsilon$ is an arbitrary number, the upper bound \eqref{Uppereta} follows and the proof is complete.
\end{proof}

 \begin{remark}
If the prior distribution is geometric \eqref{Geom}, then Theorem~\ref{Th:th2} holds whenever the parameter $\rho=\rho_\alpha\to0$ so that $|\log \rho_\alpha|/|\log\alpha| \to 0$ as $\alpha\to0$. Indeed, in this case, 
\[
\begin{aligned}
\sum_{k=0}^\infty |\log \omega_k|^r \omega_k & \le \omega_0|\log \omega_0|^r  + (1-\omega_0) \brc{ |\log(1-\omega_0)|^r +  |\log \rho_\alpha|^r + |\log (1-\rho_\alpha)|^r/\rho_\alpha}
\\
& \sim (1-\omega_0) |\log \rho_\alpha|^r.
\end{aligned}
\]
\end{remark}

\section{Asymptotic Performance of the Generalized Shiryaev--Roberts Procedure} \label{s:SR}

In the case where the prior distribution of the change point is geometric \eqref{Geom} with 
$\omega_0=0$, the statistic $R_{\omega,n}/\rho$ converges as $\rho \to0$ to the so-called Shiryaev--Roberts (SR) statistic 
\[
R_n= \sum_{k=1}^n \prod_{i=k}^n \Lambda_i = \sum_{k=1}^n \exp\set{S_n^k}, \quad R_0=0.
\]
More generally, if the prior distribution is zero-modified geometric \eqref{Geom}  with $\omega_0=\omega_0(\rho) >0$ and 
$\lim_{\rho\to0} \omega_0(\rho)/\rho = \ell$, then $R_{\omega,n}/\rho$ converges as $\rho \to0$ to the statistic 
\begin{equation}\label{GSR}
R_n^\ell= \ell \exp\set{S_n^1}  + \sum_{k=1}^n \exp\set{S_n^k}, \quad R_0^\ell = \ell
\end{equation}
that starts from $\ell$ ($\ell \ge 0$). Therefore, consider a generalized version of the SR procedure assuming that the SR statistic $R_n$ is being initialized not from 
zero but from a point $R_0=\ell$, 
which is a basis for the so-called SR-$r$ procedure introduced in \cite{Moustakidesetal-SS11,TNBbook14}.   In the present paper, we refer to this statistic as the 
generalized SR  statistic and to the corresponding stopping time
\begin{equation}\label{SRproc}
\wtT_B^\ell = \inf\set{n\ge 1: R_n^\ell \ge B}, \quad B>0
\end{equation} 
as the {\em generalized SR} (GSR) detection procedure. For the sake of brevity, we omit the superscript $\ell$  and use the notation $R_n$ and $\wtT_B$ in the following.

In contrast to the Shiryaev statistic, the GSR statistic mixes the likelihood ratios according  to the uniform improper prior but not to the given prior. 
So it is intuitively expected that the GSR procedure
is not asymptotically optimal in cases where the exponent $c>0$, but is asymptotically optimal if 
$c\to0$. Tartakovsky and Veeravalli~\cite{TartakovskyVeerTVP05} proved that this is indeed true for the conventional SR procedure with $\ell=0$ in the general non-i.i.d.\ case 
when the prior distribution is geometric. 
In this section, we show that this is true for finite-state HMMs and general prior distributions with finite mean.

Note that the GSR statistic $R_n$ is a $\Ps_\infty$-submartingale with mean $\E_\infty R_n =n+\ell$. Thus,  using the Doob submartingale inequality, we obtain
 \[
 \Ps_\infty(\wtT_B< k) =\Ps_\infty\brc{\max_{1\le i \le k-1} R_i \ge B} \le (k-1+\ell)/B, \quad k\ge 1,
 \]
so that the probability of false alarm of the SR procedure can be upper-bounded as
\begin{equation}\label{PFASR}
\PFA(\wtT_B) =\sum_{k=1}^\infty \omega_k \Ps_\infty(\wtT_B < k) \le (\bar{\nu}-1+\ell) /B ,
\end{equation}
where $\bar{\nu} = \sum_{k=1}^\infty k \, \omega_k$ is the mean of the prior distribution. Therefore, assuming that $\bar{\nu}<\infty$, 
we obtain that setting $B=B_\alpha = (\bar\nu-1+\ell)/\alpha$ implies $\wtT_{B_\alpha} \in \classA$. 
If the prior distribution is zero-modified geometric with $\omega_0=\rho \, \ell$, then $\PFA(\wtT_B) \le (1-\rho)/ \rho B$. 

The following theorem establishes asymptotic operating characteristics of the GSR procedure.

\begin{theorem}\label{Th:SRAO}
Let $r\ge 1$. Assume that conditions {\rm C1--C2} are satisfied and that, in addition, $\E_1|S_1^1|^{r+1} < \infty$. 

{\rm \bf (i)}  Suppose that the mean of the prior distribution is finite, $\bar\nu <\infty$.  If $B=B_\alpha=(\bar\nu-1+\ell)/\alpha$, then $\wtT_{B_\alpha}\in\classA$ and for all $0<m \le r$,
\begin{equation}\label{MomentsSR}
\lim_{\alpha\to0} \frac{\E^\omega[(\wtT_{B_\alpha}-\nu)^m | \wtT_{B_\alpha} \ge \nu]}{|\log \alpha|^m} =\frac{1}{\Kc^m} .
\end{equation}

{\rm \bf (ii)}  Let  the prior distribution $\{\omega_k^\alpha\}_{k\ge 0}$ satisfy condition \eqref{expon} with $c=c_\alpha\to 0$ as $\alpha\to 0$ in such a way that 
\begin{equation}\label{PriorSRhmm}
\lim_{\alpha\to 0} \frac{{\bar\nu_\alpha}}{|\log \alpha|} = 0.
\end{equation} 
Let the initial value $\ell$ be either fixed or $\ell=\ell_\alpha$ depends on $\alpha$ so that $\ell_\alpha/|\log\alpha| =o(1)$ as $\alpha\to0$.
If $B=B_\alpha=(\bar\nu_\alpha-1 +\ell)/\alpha$, then $\wtT_{B_\alpha}\in\classA$ and for all $0<m \le r$.
\begin{equation} \label{MomentsSR1}
\E^\omega[(\wtT_{B_\alpha}-\nu)^m | \wtT_{B _\alpha} \ge  \nu] \sim  \brc{\frac{|\log\alpha|}{\Kc}}^m \sim \inf_{T \in \classA} \E^\omega[(T-\nu)^m | T \ge  \nu] \quad \text{as}~ \alpha \to 0.
\end{equation}

Both assertions {\rm (i)} and {\rm (ii)} also hold if $B=B_\alpha$ is selected so that $\PFA(T_{B_\alpha}) \le \alpha$ and $\log B_\alpha\sim |\log\alpha|$ as $\alpha\to0$.
\end{theorem}

\begin{proof}
(i) First, we establish the following lower bound for moments of the detection delay of arbitrary order $m\ge 1$:
\begin{equation}\label{LowerSR}
\E^\omega[(\wtT_{B_\alpha}-\nu)^m | \wtT_{B_\alpha} \ge\nu]  \ge \brc{\frac{|\log\alpha|}{\Kc}}^m (1+o(1))  \quad \text{as}~ \alpha\to0,
\end{equation}
assuming that $\log B_\alpha \sim |\log\alpha|$. In particular, we may select $B_\alpha=(\bar\nu-1+\ell)/\alpha$.

Let $\widetilde{N}_{\alpha,\varepsilon}=(1-\varepsilon)|\log\alpha|/\Kc$. Similarly to \eqref{ADDlower} we have 
\begin{equation} \label{ADDlowerSR}
\ADD(\wtT_{B_\alpha}) \ge \widetilde{N}_{\alpha,\varepsilon}\brcs{1- \alpha - \Ps^\omega
(0 \le \wtT_{B_\alpha}-\nu < \widetilde{N}_{\alpha,\varepsilon})}.
\end{equation}
It follows from Lemma~\ref{lem1}(i) that $n^{-1}S_{k+n-1}^{k}\to\Kc$ almost surely under $\Ps_k$, and therefore, for all $\varepsilon >0$ and $k\ge 1$
\begin{equation}\label{Pmax3}
\tilde\beta_k(\alpha,\varepsilon) = \Ps_k\brc{\frac{1}{\widetilde{N}_{\alpha,\varepsilon}} \max_{1 \le n \le \widetilde{N}_{\alpha,\varepsilon}}S_{k+n-1}^k \ge (1+\varepsilon) \Kc} \to 0 \quad \text{as}~ \alpha\to0.
\end{equation}
By analogy with \eqref{ProbTle} in the proof of Lemma~\ref{Lem:A1} (see the appendix, Section~\ref{ss:AppA1}),
\begin{equation}\label{ProbTleSR}
\Ps_k\brc{0 \le \wtT_{B_\alpha}-k < \widetilde{N}_{\alpha,\varepsilon}} \le  \tilde{p}_k(\alpha,\varepsilon)  + \tilde{\beta}_k(\alpha,\varepsilon) ,
\end{equation}
where $\tilde{p}_k(\alpha,\varepsilon)$ can be shown to be upper-bounded as
\begin{equation}\label{pkle}
\tilde{p}_k(\alpha,\varepsilon) \le  \frac{k+ \ell+\widetilde{N}_{\alpha,\varepsilon}}{B_\alpha^{\varepsilon^2}}= e^{-\varepsilon^2\log B_\alpha} \brcs{k +\ell +\frac{(1-\varepsilon)|\log\alpha|}{\Kc}}.
\end{equation}
Let $K_\alpha$ be an integer that goes to infinity as $\alpha\to0$. Using \eqref{ProbTleSR} and \eqref{pkle}, we obtain
\[
\begin{aligned}
 & \Ps^\omega(0 < \wtT_{B_\alpha}-\nu < \widetilde{N}_{\alpha,\varepsilon})  =\sum_{k=0}^\infty \omega_k \Ps_k\brc{0 \le  \wtT_{B_\alpha}-k < \widetilde{N}_{\alpha,\varepsilon}}
 \\
 & \le  \Ps(\nu > K_{\alpha}) + \sum_{k=0}^{K_{\alpha}} \omega_k \Ps_k\brc{0 \le  \wtT_{B_\alpha}-k < \widetilde{N}_{\alpha,\varepsilon}}
 \\
 &  \le \Ps(\nu > K_{\alpha}) +  \sum_{k=0}^{K_{\alpha}}  \omega_k  \tilde\beta_k(\alpha,\varepsilon) 
 + e^{-\varepsilon^2\log B_\alpha} \brcs{\bar\nu +\ell +\frac{(1-\varepsilon)|\log\alpha|}{\Kc}} .
  \end{aligned}
 \]
 The first term goes to zero since $\bar\nu$ is finite. The second term goes to zero due to \eqref{Pmax3}, and the third term also goes to zero since $\lim_{\alpha\to0}\log B_\alpha/|\log\alpha|=1$. Therefore,
\[
\lim_{\alpha\to0} \Ps^\omega(0 \le  \wtT_{B_\alpha}-\nu < \widetilde{N}_{\alpha,\varepsilon})=0,
\]
and using \eqref{ADDlowerSR} and Jensen's inequality, we obtain
\[
\E^\omega[( \wtT_{B_\alpha}-\nu)^m |  \wtT_{B_\alpha} \ge\nu] \ge (1-\varepsilon)^m \brc{\frac{|\log\alpha|}{\Kc}}^m (1+o(1)).
\]
Since $\varepsilon$ can be arbitrarily small, the asymptotic lower bound \eqref{LowerSR} follows.

We now show that if $\log B_\alpha \sim |\log\alpha|$ (in particular, we may set $B_\alpha=(\bar\nu-1+\ell)/\alpha$), then for $m \le r$
\begin{equation}\label{UpperTB}
\E^\omega[(\wtT_{B_\alpha}-\nu)^m | \wtT_{B_\alpha} \ge \nu]  \le \brc{\frac{|\log\alpha|}{\Kc}}^m (1+o(1))  \quad \text{as}~ \alpha\to0
\end{equation}
whenever conditions {\rm C1--C2}  hold and $\E_1|S_1^1|^{r+1} < \infty$. This, obviously, will complete the proof of \eqref{MomentsSR}.

Observe that  for any $B>0$, $\ell \ge 0$, and $k\ge 0$,
\[
(\wtT_B-k)^+ \le \tilde\eta_B(k) = \inf\set{n \ge 1: S_{k+n-1}^k \ge \log B},
\]
Therefore, 
\[
\E^\omega[(\wtT_{B_\alpha}-\nu)^m | \wtT_{B_\alpha} \ge \nu] \le \frac{\sum_{k=0}^\infty \omega_k \E_k[\tilde\eta_{B_\alpha}(k)]^m}{1-\alpha},
\]
so that inequality \eqref{UpperTB} holds whenever 
\begin{equation}\label{UpperetaB}
\sum_{k=0}^\infty \omega_k \E_k[\tilde\eta_{B_\alpha}(k)]^r \le \brc{\frac{|\log\alpha|}{\Kc}}^r (1+o(1))  \quad \text{as}~ \alpha\to0.
\end{equation}
Define the last entry time 
\[
\hat{\tau}_{\varepsilon,k} = \sup \set{n \ge 1: n^{-1} S_{k+n-1}^k - \Kc -c < - \varepsilon} \quad (\sup\{\varnothing\} = 0) .
\]
Evidently, 
\[
(\tilde\eta_{B_\alpha}(k)-1)(\Kc+c-\varepsilon) \le S_{k+\tilde\eta_{B_\alpha}(k)-2}^k <  \log B_\alpha \quad \text{on}~ \set{\hat{\tau}_{\varepsilon,k} +1 < \tilde\eta_{B_\alpha}(k) < \infty} ,
\]
so that for every $0 < \varepsilon < \Kc+c$
\begin{equation}\label{Upperforeta}
\begin{aligned}
\brcs{\tilde\eta_{B_\alpha}(k)}^r  & \le  \brcs{1+ \frac{\log B_\alpha}{\Kc+c -\varepsilon} \Ind{\tilde\eta_{B_\alpha}(k)
> \hat{\tau}_{\varepsilon,k}+1} + (\hat{\tau}_{\varepsilon,k}+1) \Ind{\tilde\eta_{B_\alpha}(k) \le \hat{\tau}_{\varepsilon,k}}}^r
\\
& \le  \brc{\frac{\log B_\alpha}{\Kc+c-\varepsilon} +\hat{\tau}_{\varepsilon,k}+2}^r.
\end{aligned}
\end{equation}
By \eqref{cas}--\eqref{cas1},
\[
\sum_{k=0}^\infty \omega_k \E_k(\hat\tau_{\varepsilon, k})^r  <\infty \quad \text{for all}~ \varepsilon >0, 
\]
so using condition \eqref{Prior3hmm} and the fact that $\log B_\alpha\sim |\log\alpha|$, we conclude that for any $0<\varepsilon < \Kc$
\[
\sum_{k=0}^\infty \omega_k \E_k[\tilde\eta_{B_\alpha}(k)]^r \le \brc{\frac{|\log\alpha|}{\Kc-\varepsilon}}^r (1+o(1))  \quad \text{as}~ \alpha\to 0.
\]
Since $\varepsilon$ can be arbitrarily small, the upper bound \eqref{UpperetaB} follows, which implies the upper bound \eqref{UpperTB}.

(ii) Recall first that in the proof of Theorem~\ref{Th:th2} we established the asymptotic lower bound \eqref{Lowergendalpha}, which holds as long as conditions C1 and C2 are satisfied. 
Therefore, we need only to prove the asymptotic upper bound \eqref{UpperTB} under conditions postulated in (ii). Similarly to \eqref{Upperforeta} we have
\begin{equation*}
[(\wtT_{B_\alpha}-k)^+]^r \le  \brc{\frac{\log B_\alpha}{\Kc+c_\alpha-\varepsilon} +\hat{\tau}_{\varepsilon,k}+2}^r = 
\brc{\frac{\log [(\bar\nu_\alpha-1+\ell)/\alpha]}{\Kc+c_\alpha-\varepsilon} +\hat{\tau}_{\varepsilon,k}+2}^r.
\end{equation*}
Again, by \eqref{cas}--\eqref{cas1},
\[
\sum_{k=0}^\infty \omega_k^\alpha \E_k(\hat\tau_{\varepsilon, k})^r  <\infty \quad \text{for all}~ \varepsilon >0.
\]
Using condition \eqref{PriorSRhmm} and the fact that $\lim_{\alpha\to0} \ell/|\log\alpha| =0$, we obtain that for any $0<\varepsilon < \Kc$
\[
\sum_{k=0}^\infty \omega_k^\alpha [(\wtT_{B_\alpha}-k)^+]^r \le \brc{\frac{|\log\alpha|}{\Kc-\varepsilon}}^r (1+o(1))  \quad \text{as}~ \alpha\to 0.
\]
Since $\varepsilon$ can be arbitrarily small, the upper bound \eqref{UpperTB} follows and the proof is complete.
\end{proof}

Theorem~\ref{Th:SRAO} shows that the GSR procedure is asymptotically optimal as $\alpha\to0$ in class
$\classA$, minimizing moments of the detection delay up to order $r$, 
only for heavy-tailed priors when $c=0$ or 
for priors with exponential tails $(c>0$) when the exponent $c=c_\alpha$ vanishes as $\alpha\to0$. As mentioned above, this is expected since the GSR procedure exploits the uniform 
improper prior distribution of the change point over positive integers.

\section{ Higher Order Asymptotic Approximations to the Average Detection Delay }\label{s:HOAO}

Note that when $m=1$ in \eqref{Momentsgen} and \eqref{MomentsSR}, we obtain the following first-order asymptotic approximations to the average detection delay of the Shiryaev and 
GSR procedures
\begin{equation}\label{ADDgen}
\ADD(T_A) = \brc{\frac{\log A}{\Kc + c}} (1+o(1)), \quad \ADD(\wtT_B) = \brc{\frac{\log B}{\Kc}} (1+o(1)) \quad  \text{as $A, B \to \infty$}.
\end{equation}
These approximations hold as long as conditions C1--C2 and the moment condition 
$\E_1|S_1^1|^{2} < \infty$ are satisfied. 

In this section, we  derive high-order approximations to the ADD of these procedures up to a vanishing term $o(1)$ based on the Markov nonlinear renewal theory, 
assuming that the prior distribution of the change point is  zero-modified geometric  \eqref{Geom}. 

\subsection{The Shiryaev Procedure}\label{HOAShiryaev}

 Define the statistic $R_n^\rho= R_{n,\omega}/\rho$, which is given by the recursion 
\begin{equation}\label{Rrecursion}
R_n^\rho= (1+R_{n-1}^\rho) \Lambda_n^\rho, \quad n \ge 1, ~~ R_0^\rho= \omega_0/(1-\omega_0)\rho,
\end{equation}
where 
\[
\Lambda_n^\rho= \frac{e^{g(W_{n-1},W_n)}}{1-\rho} = e^{g_\rho(W_{n-1},W_n)}, \quad g_\rho(W_{i-1},W_{i}) = g(W_{i-1},W_{i}) + |\log(1-\rho)| . 
\]
Obviously, the Shiryaev procedure can be written as
\[
T_A = \inf\set{n\ge 1: R_n^\rho \ge A/\rho}.
\]

Note that we have 
\begin{align*}
R_{n}^\rho & = (1 + R_{0}^\rho) \prod_{i=1}^{n} \Lambda_i^\rho + \brc{\prod_{i=1}^{n} \Lambda_i^\rho}  \sum_{j=1}^{n-1} \prod_{s=1}^{j} (\Lambda_s^\rho)^{-1} \\
& =\brc{1 + R_{0}^\rho + \sum_{j=1}^{n-1} (1-\rho)^{j} e^{-\sum_{s=1}^j g(W_{s-1},W_{s})}} \,  \prod_{i=1}^{n} \Lambda_i^\rho ,
\end{align*}
and hence,
\begin{equation} \label{logRn}
\log R_{n}^\rho = \sum_{i=1}^{n} g_\rho(W_{i-1},W_{i}) + \log \brc{1 + R_{0}^\rho + V_{n}} ,
\end{equation}
where 
\[
V_{n} = \sum_{j=1}^{n-1}  (1-\rho)^{j} e^{-\sum_{s=1}^j g(W_{s-1},W_{s})}.
\]

Let $a=\log(A/\rho)$. Clearly, the stopping time $T_A=T_a$ can be equivalently written as
\[
T_a =\inf\set{n\ge 1: \log R_{n}^\rho \ge a},  
\]
which by \eqref{logRn} can be also written as 
\begin{equation}\label{Tn1}
T_a = \inf \{n\ge 1: \widetilde{S}_{n}^\rho + \eta_{n} \ge a\},
\end{equation}
where $\eta_{n} = \log \brc{1+ R_{0}^\rho +V_{n}}$  and $\widetilde{S}_{n}^{\rho} =\sum_{i=1}^{n} g_\rho(W_{i-1},W_{i})=S_n +n |\log (1-\rho)|$, $n\ge 1$.  Here $S_n$ denotes the partial sums $\sum_{i=1}^{n} g(W_{i-1},W_{i})$, $n \ge 1$. 
Note that in \eqref{Tn1} the initial condition $W_0$ can be an arbitrary fixed number or a random variable.
Obviously, $\{\widetilde{S}_{n}^{\rho}\}_{n\ge 1}$  is a Markov random walk with stationary 
mean $\E^{\Pi} \widetilde{S}_{1}^{\rho} = \Kc+|\log (1-\rho)|$, where   $\E^{\Pi}(\cdot)= \int \E_1(\cdot | W_0=w) d\Pi(w)$ denotes the expectation of the extended Markov chain $\{W_n\}_{n\ge 0}$ 
under the invariant measure $\Pi$. Let $\chi_a = \widetilde{S}_{n}^{\rho} + \eta_{T_a} -a$ be a corresponding overshoot. Then we have
\begin{equation}\label{ExpS}
\E_1 [\widetilde{S}^{\rho}_{T_a}|W_{0}=w] = a + \E_1 [\chi(a)|W_{0}=w] - \E_1[\eta_{T_a}|W_{0}=w].
\end{equation}

For $b>0$, define
\begin{equation}\label{5.5}
N_b=\inf \{n \ge 1: \widetilde{S}_{n}^{\rho} \ge b\},
\end{equation}
and let $\kappa_b= \widetilde{S}_{n}^{\rho} - b$ (on $\{N_b < \infty\}$) 
denote the overshoot of the statistic
$\widetilde{S}_{n}^{\rho}$ crossing the threshold $b$ at time $n=N_b$. When $b=0$, 
we denote $N_b$ in (\ref{5.5}) as $N_+$. For a given $w \in {\cal X}$, let
\begin{equation}
G(y,\rho, \Kc)=\lim_{b\rightarrow \infty} {\Ps}_1 \{ \kappa_b \le y |W_0= w \}
\end{equation}
be the limiting distribution of the overshoot. Note that this distribution does not depend on $w$.  

To approximate the expected value of $\kappa_b$, we need the following notation first.
Let ${\Ps}_{1,+}(w, B) = {\Ps}_{1,+}\{W_{N_b} \in B | W_0 = w \}$ denote the transition probability associated with the
Markov chain $\{W_n, n \ge 0\}$ generated by the ascending ladder variable
$\widetilde{S}_{n}^{\rho}$.  Under the $V$-uniform ergodicity condition (to be proved in the appendix) and $\E^\pi Y_1> 0$, a similar argument as on page 255 of Fuh and Lai \cite{FuhLai01}
yields that the transition probability ${\Ps}_{1,+}(w,  \cdot)$ has an invariant measure
$\Pi_+$. Let ${\E}^{\Pi_+}$ denote expectation under $W_0$ having the distribution $\Pi_+$.

It is known  that
\[ 
\lim_{b\rightarrow \infty} \E_1 (\kappa_b|W_0= w)
 = \int_0^\infty y \, \d G(y,\rho, \Kc) = \frac{{\E}^{\Pi_+} 
 (\widetilde{S}_{N_+}^{\rho})^2}{2 {\E}^{\Pi_+} \widetilde{S}_{N_+}^{\rho}}.
 \]
(cf. Theorem 1 of Fuh \cite{FuhAS04} and Proposition~\ref{P2} in the appendix).

Let us also define
\[ 
\zeta(\rho,\Kc) =\lim_{b\rightarrow \infty} \E_1 ( e^{- \kappa_b}|W_0= w )  = \int_0^\infty e^{-y} \, \d G(y,\rho, \Kc),
 \]
and 
\begin{equation}\label{ConstantC}
C(\rho,\Kc)=\E_1\set{\log \brcs{1+R_0^\rho+ \sum_{k=1}^{\infty}(1-\rho)^k e^{-S_k}}}.
\end{equation}

Note that by \eqref{Tn1},
\[ 
\widetilde{S}_{T_a}^{\rho} = a - \eta_{T_a} + \chi_a~~~\mbox{ on } \{T_a < \infty\},  
\]
where $\chi_a= \widetilde{S}_{T_a}^{\rho} + \eta_{T_a}-a$ is the overshoot of 
$\widetilde{S}_{n}^{\rho} + \eta_n$
crossing the boundary $a$ at time $T_a$. Taking the expectations on both sides, using \eqref{ExpS} and applying Wald's identity for products
of Markovian random matrices (cf.\ Theorem 2 of Fuh \cite{FuhAS03}), we obtain 
\begin{eqnarray}\label{4.7}
&~& (\Kc+|\log (1-\rho)|) \E_1( T_a|W_0= w)  + \int_{\cal X} \Delta(w') \, \d \Pi_+(w')- \Delta(w) \\
&~&~~~~~~= \E_1 ( S_{T_a}^\rho |W_0= w ) = a- \E_1 ( \eta_{T_a} |W_0= w )+ \E_1( \chi_a|W_0= w), \nonumber
\end{eqnarray}
where $\Delta : {\cal X} \rightarrow  R$ solves the Poisson equation 
\begin{eqnarray}\label{4.8}
\E_1 \brcs{\Delta(W_1)|W_0=w} - \Delta(w) = \E_1\brcs{S_1^\rho | W_0=w}  - \E^\Pi S_1^\rho
\end{eqnarray}
for almost every $w \in {\cal X}$ with $\E^\Pi \Delta (W_1) = 0$. 

The crucial observations are that the sequence $\{\eta_n, n\ge 1\}$ is slowly changing
and that $\eta_n$ converges $\Ps_1$-a.s. as $n \rightarrow \infty$ to the random variable
\begin{eqnarray}\label{4.9}
\eta =\log \left\{1+R_0^\rho+\sum_{k=1}^{\infty} (1-\rho)^{k} e^{- S_k} \right\}
\end{eqnarray}
with finite expectation ${\E}^{\Pi_+} \eta=C(\rho,\Kc)$, where $C(\rho,\Kc)$ is defined in \eqref{ConstantC}.  
An important consequence of the slowly changing property is that, under mild conditions,
the limiting distribution of the overshoot of a Markov random walk over a fixed threshold 
does not change by the
addition of a slowly changing nonlinear term (see Theorem 1 in Fuh \cite{FuhAS04}).

The mathematical details are given in Theorem~\ref{th3} below.
More importantly, Markov nonlinear renewal theory allows us to obtain an approximation
to $\mbox{PFA}(T_A)$, the probability of false alarm, that takes the overshoot into account. This approximation is
useful for practical applications where the value of $\Kc$ is moderate. (For small values of $\rho$ and $\Kc$ the
overshoot can be neglected.)

\begin{theorem}\label{th3} 
Let $Y_0,Y_1,\cdots,Y_n$ be a sequence of 
random variables from a hidden Markov model $\{Y_n, n \ge 0\}$. Assume {\rm C1--C2} hold. 
Let the prior distribution of the change point $\nu$ be the zero-modified geometric distribution \eqref{Geom}, 
and assume that ${S}_1$ is nonarithmetic with respect to $\Ps_{\infty}$ and $\Ps_1$. 

{\rm \bf (i)} If $0< \Kc < \infty$, then as $A \rightarrow \infty$ 
\begin{equation} \label{HOPFA}
\pfa(T_A)=\frac{\zeta(\rho,\Kc)}{A} (1+o(1)).
\end{equation}

{\rm \bf (ii)} If, in addition, the second moment  of the log-likelihood ratio is finite,
${\E}_1|S_1|^2 < \infty$, then for $w \in {\cal X}$, as $A \rightarrow \infty$ 
\begin{equation}\label{4.10}
\begin{split}
 {\E}_1 (T_A |W_0=w) &= \frac{1}{\Kc +|\log (1-\rho)|}
\bigg(\log \frac{A}{\rho}- {\E}^{\Pi_+}  \eta + \frac{{\E}^{\Pi_+} (\widetilde{S}_{N_+}^\rho)^2}
{2 {\E}^{\Pi_+} \widetilde{S}_{N_+}^\rho} \\
& \quad - \int_{\cal S} \Delta(\tilde{w}) \, \d \Pi_+(\tilde{w}) +\Delta(w)\bigg) + o(1). 
\end{split}
\end{equation}
\end{theorem}

\begin{remark} The constants ${\E}^{\Pi_+} (\widetilde{S}_{N_+}^\rho)^2/2 {\E}^{\Pi_+} 
\widetilde{S}_{N_+}^\rho$
and ${\E}^{\Pi_+} \eta=C(\rho,\Kc)$ are the subject of the Markov nonlinear renewal theory.
The constant $-\int_{\cal S} \Delta(\tilde{w}) d\Pi_+(\tilde{w}) + \Delta(w)$ is due to Markovian
dependence via Poisson equation \eqref{4.7}.
\end{remark}

\begin{proof}
(i) By \eqref{post_R}, $1-\Ps(\nu \le T_A|\Fc_{T_A}) = (1+R_{T_A,\omega})^{-1} = (1+\rho R_{T_A}^\rho)^{-1}$ and we have
\[
\PFA(T_A) = \E^\omega[1-\Ps(\nu \le T_A|\Fc_{T_A})] = \E^\omega\brcs{1+\rho A (R_{T_A}^\rho/A)} = \frac{1}{\rho A} \E^\omega \brcs{e^{-\chi_a}}(1+o(1))
\]
as $A \rightarrow \infty$, where $\chi_a = S_{T_a}^\rho + \eta_{T_a}-a$. Since $\chi_a  \ge 0$ and $\PFA(T_A)< 1/A$, it follows that 
\begin{align*}
\E^\omega \brcs{e^{-\chi_a}} &= \E^\omega \brc{e^{-\chi_a} |T_A < \nu} \PFA(T_A)
+\E^\omega \brc{e^{-\chi_a}|T_A \ge \nu}(1-\PFA(T_A)) 
\\
&= \E^\omega \brc{e^{-\chi_a} | T_A \ge \nu} +O(1/A)
\end{align*}
 as $A \rightarrow \infty$. Therefore, it suffices to evaluate the value of 
\[
\E^\omega \brc{e^{-\chi_a}|T_A \ge \nu} = \sum^\infty_{k=1}P\{\nu = k |T_A \ge k\} \E^\omega \brc{e^{-\chi_a} | T_A \ge k}.
\]
To this end, we recall that, by \eqref{Tn1} 
\[
T_a = \{ n \geq 1 : \widetilde{S}_n^\rho + \eta_{n} \ge a \},
\]
where $ \widetilde{S}_n^\rho = S_n +n| \log(1-\rho)|$ is a Markov random walk with the expectation $\Kc+ |\log(1-\rho)|$ and 
$\eta_n$, $n\ge 1$  are slowly changing under $\Ps_1$. Since, by conditions 
C1, $0< \Kc <\infty$, we can apply Theorem 1 of \cite{FuhAS04}, see also Proposition~\ref{P2} in the appendix, to obtain
\[
\lim_{A\rightarrow \infty}\E_1 \brc{e^{-\chi_a} | T_A \ge k} = \int _0^\infty e^{-y} \, \d G(y,\rho,{\cal K}) = \zeta(\rho,{\cal K}).
\] 
Also,
\[
\lim_{A\rightarrow \infty} \Ps\brc{ \nu = k| T_A \ge k} = \lim_{A \rightarrow \infty} \frac{\omega_k \Ps_\infty(T_A \ge k)}{\Ps^\omega (T_A \ge \nu)}= \omega_k,
\]
so that 
\[
\lim_{A\rightarrow \infty}\E_k \brc{e^{-\chi_a} | T_A \ge \nu} = \lim_{A \rightarrow \infty}\E_k e^{-\chi_a} = \zeta(\rho,{\cal K}),
\]
which completes the proof of (\ref{HOPFA}).

(ii)  The probability $\Ps_1$ and expectation $\E_1$ in the proof below are taken under $W_0=w$, 
i.e., $\Ps_1(\cdot|W_0=w)$ and $\E_1(\cdot|W_0=w)$. We omit conditioning on $W_0=w$ for brevity. 
The proof of \eqref{4.10} is based on the Markov nonlinear renewal theory (see Theorem 3 and Corollary 1 in
Fuh \cite{FuhAS04}). A simplified version suitable in our case is given in the appendix (see  Proposition~\ref{P3}). 

By \eqref{Tn1}, the stopping time $T_A=T_a$ is based on thresholding the sum of the Markov random walk
$S_n^\rho$ and the nonlinear term $\eta_n$. Note that
\[ 
\eta_n \longrightarrow_{\scriptstyle n\rightarrow \infty} \eta~~~\Ps_{1}\mbox{-a.s.} ~~~\mbox{and}~~~
\E_{1} \eta_n \longrightarrow_{\scriptstyle n\rightarrow \infty} \E_{1} \eta,
\]
and $\eta_n$, $n\ge 1$ are slowly changing under $\Ps_{1}$. In order to apply Proposition~\ref{P3} (see appendix) in our case, we have to check the following three 
conditions:
\begin{align}
& \sum_{n=1}^\infty \Ps_1 \{\eta_n \le -\varepsilon n\} < \infty ~~ \text{ for some}~ 0< \varepsilon < \Kc; \label{Cond1} \\
& \max_{0\le k\le n}|\eta_{n+k}|,~ n\ge 1,~ \text{are}~ \Ps_1\text{-uniformly integrable};  \label{Cond2}\\
& \lim_{A \rightarrow \infty} (\log A)~\Ps_1\set{T_A \le \frac{\varepsilon \log A}{\Kc+|\log(1-\rho)|}}=0~~ \text{for some} ~0<\varepsilon<1. \label{Cond3}
\end{align}

Condition \eqref{Cond1} obviously holds because $\eta_n  \ge 0$.  Condition \eqref{Cond2} holds because  $\eta_{2n}$, $n\ge 1$
are $\Ps_1$-uniformly integrable since $\eta_{2n}\le \eta$ and $\E_1 \eta < \infty$ and $\max_{0\le k\le n}|\eta_{n+k}|=\eta_{2n}$ ($\eta_n$, $n=1,2\cdots$ are nondecreasing). 

To verify condition \eqref{Cond3} we now prove that for all $A>0$ and $0<\varepsilon <1$
\begin{equation}\label{ProbTAle}
\Ps_1\brc{T_A \le N_{A,\varepsilon}} \le  \frac{1}{1-\rho} \exp\set{-\frac{\Kc~\varepsilon^2}{\Kc + |\log (1-\rho)} \, \log A} + \beta(A,\varepsilon) ,
\end{equation}
where 
\[
N_{A,\varepsilon} = \frac{(1-\varepsilon) \log A}{\Kc+\log(1-\rho)} \quad \text{and} \quad 
\beta(A,\varepsilon) =\Ps_1\brc{\frac{1}{N_{A,\varepsilon}} \max_{1 \le n \le N_{A,\varepsilon}}  S_{n} \ge (1+\varepsilon) \Kc}. 
\]
Moreover, we will establish that under the second moment condition $\E_1 |S_1|^2 <\infty$, the probability $\beta(A,\varepsilon)$ vanishes as $A\to\infty$ faster than $1/\log A$. This implies that
$\Ps_1\brc{T_A \le N_{A,\varepsilon}}=o(1/\log A)$ as $A\to\infty$, i.e., condition \eqref{Cond3}.

To obtain the inequality \eqref{ProbTAle} we follow the proof of Lemma~\ref{Lem:A1} in the appendix, replacing $|\log\alpha|$ by $\log A$, setting $k=1$ and $\delta_\alpha=0$,
 and noting that for the zero-modified geometric prior $\log \Ps(\nu > N_{A,\varepsilon}) = N_{A,\varepsilon} \log(1-\rho)$.  Then using \eqref{ProbTle}--\eqref{UpperApk}, we obtain that the 
 inequality \eqref{ProbTAle} holds for all $A>0$ and $0<\varepsilon <1$.  Thus, it remains to prove that $\lim_{A\to\infty} [\beta(A,\varepsilon) \log A]=0$. 
 By Proposition~\ref{prop1}, we can apply Theorem 6 of Fuh and Zhang \cite{FuhZhang00}, which yields that if $\E_1 |S_1|^2 < \infty$, then for all $\varepsilon>0$ 
\begin{equation*}
\sum_{n=1}^\infty \Ps_1 \left\{\max_{1\le k\le n}(S_k- k \Kc ) \ge \varepsilon n\right\} < \infty.
\end{equation*}
This implies that the summand is $o(1/n)$ for a large $n$. Since
 \[
 \begin{aligned}
 \beta(A,\varepsilon)  =\Ps_1\brc{\max_{1 \le n \le N_{A,\varepsilon}}  (S_{n} - N_{A,\varepsilon} \Kc) \ge \varepsilon \Kc N_{A,\varepsilon}}
  \le \Ps_1\brc{\max_{1 \le n \le N_{A,\varepsilon}}  (S_{n} - n \Kc) \ge \varepsilon \Kc N_{A,\varepsilon}}
 \end{aligned}
 \]
it follows that $\beta(A,\varepsilon)=o(1/\log A)$ as $A\to\infty$. This implies condition \eqref{Cond3}. 

Applying Proposition~\ref{P3}  (see Appendix) completes the proof. 
\end{proof}

\subsection{The Generalized Shiryaev--Roberts Procedure}\label{HOAGSR}

Since the GSR procedure $\wtT_B$ defined in \eqref{GSR} and \eqref{SRproc} is a limit of the Shiryaev procedure as the parameter of the geometric prior distribution $\rho$ goes to zero,
it is intuitively obvious that the higher order approximation to the conditional average detection delay $\E_1(\wtT_B|W_0=w)$ is given by \eqref{4.10} with $\rho=0$ and $A/\rho=B$. 
This is indeed the case as the following theorem shows. The proof of this theorem is essentially similar to the proof of Theorem~\ref{th3}(ii) and for this reason, it is omitted.

\begin{theorem}\label{th6} 
Let $\{Y_n, n \ge 0\}$ be a hidden Markov model. Assume conditions {\rm C1--C2} hold. 
Let the prior distribution of the change point be the zero-modified geometric distribution \eqref{Geom}. 
Assume that ${S}_1$ is nonarithmetic with respect to $\Ps_{\infty}$ and $\Ps_1$ and that $\E_1|S_1|^2 < \infty$. Then for $w \in {\cal X}$,
as $B \rightarrow \infty$
\begin{equation}
\begin{split}\label{HOAOGSR}
 \E_1 (\wtT_B |W_0=w) &= \frac{1}{\Kc} \bigg(\log B - {\E}^{\Pi_+}  \tilde{\eta} +  
 \frac{{\E}^{\Pi_+} (\widetilde{S}_{N_+}^\rho)^2}{2 {\E}^{\Pi_+} \widetilde{S}_{N_+}^\rho}  \\
& \quad  \quad  \quad - \int_{\cal S} \Delta(\tilde{w}) \, \d \Pi_+(\tilde{w}) +\Delta(w)\bigg) + o(1), 
\end{split}
\end{equation}
where $\tilde{\eta} = \log (1+ \ell + \sum_{j=1}^\infty e^{-S_j})$ and 
$\ell= \lim_{\rho \to \infty} w_0(\rho)/\rho$.
\end{theorem}

\section{ Examples}\label{s:Examples}

Below we consider several examples that are of interest in certain applications (see, e.g., \cite{Willett, RGT2013,Raghavanetal-IEEE2014}). 
While our main concern in this paper is checking conditions under which the Shiryaev detection 
procedure is asymptotically optimal, another
important issue is the feasibility of its implementation, i.e., computing the Shiryaev decision statistic. We address both issues in each example.  

In many applications, including Examples 1 and 2 considered below, one is interested in a simplified HMM where the observations $Y_n$, $n=1,2,\dots$ 
are conditionally independent, conditioned on the Markov chain $X_n$, i.e.,
$f_j (Y_n | X_n, Y_{n-1}) = f_j (Y_n | X_n)$. In this particular case, the conditions C1 and C2 in the above theorems can be simplified to the following conditions.

{\bf C1$^{'}$}. For each $j=\infty,0$, the Markov chain $\{X_n, n \ge 0
\}$ defined in (\ref{ssm}) and (\ref{ssmden}) is ergodic 
(positive recurrent, irreducible and aperiodic) on a finite state space ${\cal X}=\{1,\cdots,d\}$ and has stationary probability
 $\pi$.

{\bf C$2^{'}$}. The Kullback--Leibler information number is positive and finite, $0< \Kc < \infty$. For each $j=\infty,0$, the random matrices $M_0^j$ and
$M_1^j$, defined in (\ref{rm1}) and (\ref{rm2}), are invertible $\Ps_j$ almost surely and for some $r > 0$,
\begin{equation}\label{C2}
 \left | \sum_{x_0,x_1 \in {\cal X}} \int_{y \in {\bf R}^d} \pi_{j}(x_0)  p_j(x_0,x_1) y^{r+1} f_j(y|x_1)Q(dy) \right  |  < \infty.
\end{equation}

\subsection{Example 1: Target Track Management Application}\label{ss:Ex1}

We begin with an example which is motivated by certain multisensor target track management applications \cite{Willett} that are discussed later on at the end of this subsection.

Let $X_n \in \{1,2\}$ be a two-state (hidden, unobserved) Markov chain with transition probabilities 
$\Ps_j(X_n =1|X_{n-1}=2)=p_j(2,1)=p$ and $\Ps_j(X_n =2|X_{n-1}=1)=p_j(1,2)=q$, $n \ge 1$ and initial stationary distribution 
$\Ps_j(X_0=1)=\pi_j(1)=p/(p+q)$ for both $j=\infty$ and $j=0$. 
Under the pre-change hypothesis $H_\infty:\nu=\infty$,  the conditional density of the observation $Y_n$ is
\[
p(Y_n|\Yb^{n-1}_0, X_n=l, H_\infty) = g_l(Y_n) \quad \text{for}~l=1, 2,
\]
and under the hypothesis $H_k: \nu=k$, the observations $Y_k, Y_{k+1}, \dots$ are i.i.d.\ with density $f(y)$.
The pre-change joint density of the vector $\Yb_0^n$ is
\[
p_\infty(\Yb^n_0) = \prod_{i=1}^n p_\infty(Y_i|\Yb^{i-1}_0),
\]
where 
\begin{equation} \label{f0cond}
p_\infty(Y_i|\Yb_0^{i-1})  = \sum_{l=1}^2 g_l(Y_i) \Ps(X_i=l | \Yb^{i-1}_0), \quad i \ge 1
\end{equation}
and the posterior probability $\Ps(X_i=l |\Yb^{i-1}_0)=P_{i|i-1}(l)$ is obtained by a Bayesian update as
follows. By the Bayes rule, the posterior probability $\Ps(X_i =l |\Yb^{i}_0):=P_i(l)$ is given by
\begin{equation} \label{posterior}
P_i(l)= \frac{g_l(Y_i) P_{i|i-1}(l)}{\sum_{s=1}^2 g_s(Y_i)P_{i|i-1}(s)}.
\end{equation}
The probability $P_{i|i-1}(X_i)$ is used as the prior probability for the update (prediction term) and can be computed as 
\begin{equation} \label{prediction}
P_{i|i-1}(2) = P_{i-1}(2) (1-p) + P_{i-1}(1) q, \quad P_{i|i-1}(1) = P_{i-1}(1) (1-q) + P_{i-1}(2)p.
\end{equation}

The statistic $R_{n,\omega}$ defined in \eqref{Rnp} can be computed recursively as
\begin{equation}\label{Shirstatrec}
R_{n,\omega} = (R_{n-1,\omega} + \omega_{n,n}) \Lambda_n, \quad n \ge 1, ~~ R_{0,\omega} = \omega_0/(1-\omega_0),
\end{equation}
where the likelihood ratio ``increment'' $\Lambda_n = f(Y_n)/p_\infty(Y_n|\Yb^{n-1}_0)$
can be effectively computed using \eqref{f0cond}, \eqref{posterior}, and \eqref{prediction}. 
Here $\omega_{k,n}$ is defined in (\ref{wkn}). Therefore, in this example, the computational cost of the Shiryaev rule is small, and it can be easily implemented on-line.

Condition C1$^{'}$ obviously holds. Condition \eqref{C2} in C2$^{'}$ holds if 
\begin{align*}
\int_{-\infty}^\infty |y|^{r+1} f(y) \, \d Q(y) <\infty \quad \text{and}\quad \int_{-\infty}^\infty |y|^{r+1} g_l(y) \, \d Q(y) < \infty ~~ \text{for} ~ l=1,2
\end{align*}
since
\begin{align*}
 \abs{\sum_{i,l=1,2}  \int_{-\infty}^{\infty}  \pi_0(i)  p_0(i,l) y^{r+1}  f(y) \, \d Q(y)} \le  \int_{-\infty}^{\infty} |y|^{r+1}  f(y) \, \d Q(y)
\end{align*}
and
\begin{align*}
  \abs{\sum_{i, l=1,2}  \int_{-\infty}^{\infty}  \pi_\infty(i)  p_\infty(i,l) y^{r+1}  g_{l}(y) \, \d Q(y)} 
 \le \sum_{i,l=1,2}  \pi_\infty(i)  p_\infty(i,l)  \int_{-\infty}^{\infty}  |y|^{r+1}  g_l(y) \, \d Q(y) .
 \end{align*}
In this case, the Kullback--Leibler number $\Kc$ is obviously finite. Therefore, the Shiryaev detection rule is nearly optimal, minimizing asymptotically moments of the detection delay up to order $r$. In particular, if $f(y)$ and $g_l(y)$ are Gaussian densities, then the Shiryaev procedure minimizes  all positive moments of the delay to detection.

In \cite{Willett}, this problem was considered in the context of target track management, specifically for termination of tracks from targets with drastically fluctuating signal-to-noise ratios in active sonar systems. 
This drastic fluctuation was proposed to model as Markovian switches between low and high intensity signals, which lead to low and high probabilities of detection. In this scenario, 
one is particularly interested in the Bernoulli model where $Y_n=0,1$ and
\[
g_l(Y_n) = (P_d^l)^{Y_n}(1-P_d^l)^{1-Y_n}, ~~ l=1, 2; \quad f(Y_n)=(P_{fa})^{Y_n}(1-P_{fa})^{1-Y_n} ,
\]
where $P_d^1=P_d^{\rm High}$ and $P_d^2=P_d^{\rm Low}$ are local probabilities of detection (in single scans) for high and low intensity signals, respectively, 
and $P_{fa}$ is the probability of a false alarm that satisfy inequalities $P_d^1 >P_d^2> P_{fa}$. The target track is terminated at the first time the statistic $R_{n,\omega}$ exceeds threshold $A$.
Note that in this application area, the results of simulations presented in \cite{Willett} show that the Shiryaev procedure  performs very well while popular Page's CUSUM procedure performs poorly. 
Also, since for the Bernoulli model all moments are finite, the Shiryaev procedure minimizes asymptotically all positive moments of the detection delay. 

\subsection{Example 2: A two-state HMM with i.i.d.\ observations} 

Consider a binary-state case with i.i.d.\ observations in each state. Specifically, let $\theta$ be a parameter taking two possible values $\theta_0$ and $\theta_1$ and  let $X_n\in\{1,2\}$ be a two-state ergodic 
Markov chain with the transition matrix   
\begin{eqnarray*}
[p_\theta(i,l)] = \left[
\begin{array}{cc}
1 - p_{\theta} &  p_{\theta} \\
q_{\theta} & 1 - q_{\theta}
\end{array}
\right] 
\end{eqnarray*}
and stationary initial distribution ${\sf P}_{\theta}(X_0 = 2)  = 1-{\sf P}_{\theta}(X_0 = 1)= \pi_\theta(2)=q_{\theta}/(p_{\theta} + q_{\theta})$ for some $\{p_{\theta} , q_{\theta} \} \in [0 , 1]$. 
Further, assume that conditioned on $X_n$ 
the observations $Y_n$  are i.i.d.\ with densities $f_\theta(y| X_n=l) =f_\theta^{(l)}(y)$ ($l=1,2$), where the parameter $\theta=\theta_0$ pre-change and $\theta=\theta_1$ post-change.  
In other words, in this scenario, the conditional density
\begin{equation*}
p( Y_n |\Yb^{n-1}_0) = \begin{cases}
p_{\theta_0} (Y_n|\Yb^{n-1}_0) & \text{if} ~ n \le \nu - 1 \\
p_{\theta_1} (Y_n|\Yb^{n-1}_0) & \text{if} ~ n \ge \nu 
\end{cases} ,
\end{equation*}
which in terms of the joint density of $\Yb^n_0$ yields
\begin{equation*}
p(\Yb^n_0) = \begin{cases}
p_{\theta_0} (\Yb^n_0) & {\rm if} ~ n \le \nu - 1 \\
p_{\theta_1}(\Yb^n_0) \cdot \frac{p_{\theta_0}(\Yb^{\nu-1}_0)}{p_{\theta_1}(\Yb^{\nu-1}_0) }  & {\rm if} ~ n \ge \nu 
\end{cases} .
\end{equation*}
Thus,  the increment of the likelihood ratio $\Lambda_n = LR_n/LR_{n-1}$ does not depend on the change point and, as a result, the Shiryaev detection statistic obeys the recursion 
\eqref{Shirstatrec}, so that in order to implement the Shiryaev procedure it suffices to develop an efficient computational scheme
for the likelihood ratio $LR_n=p_{\theta_1}(\Yb^n_0)/p_{\theta_0}(\Yb^n_0)$. To obtain a recursion for $LR_n$, define the probabilities $P_{\theta, n}  :=  {\sf P}_{\theta} ( \Yb^n_0, X_n = 2)$ and 
$\widetilde{P}_{\theta, n}  :=  {\sf P}_{\theta} ( \Yb^n_0, X_n = 1)$. Straightforward argument shows that for $n \ge 1$
\begin{align}
P_{\theta, n} & =   \left[ P_{\theta, n-1} \, p_\theta(2,2) + \widetilde{P}_{\theta, n-1} \, p_\theta(1,2) \right] \, f_{\theta}^{(2)}(Y_n) ; \label{Rec1}
\\
\widetilde{P}_{\theta, n}  & =  \left[ P_{\theta, n-1} \, p_\theta(2,1) + \widetilde{P}_{\theta, n-1} \, p_\theta(1,1) \right] \, f_{\theta}^{(1)}(Y_n) \label{Rec2}
\end{align}
with $P_{\theta, 0}  = \pi_\theta(2)$ and $\widetilde{P}_{\theta,0} =\pi_\theta(1) =1-\pi_\theta(2)$. Since $p_{\theta} (\Yb^n_0)=P_{\theta, n}+\widetilde{P}_{\theta,n}$ we obtain that
\[
LR_n = \frac{ P_{\theta_1, n} + \widetilde{P}_{\theta_1, n} } {P_{\theta_0, n} + \widetilde{P}_{\theta_0, n} }. 
\]
Therefore, to implement the Shiryaev (or the SR) procedure we have to update and store $P_{\theta, n}$ and $\widetilde{P}_{\theta, n}$ for the two parameters values $\theta_0$ and $\theta_1$ using simple 
recursions \eqref{Rec1} and \eqref{Rec2}.

Condition C1$^{'}$ obviously holds. Assume that the observations are Gaussian with unit variance and different mean values in 
pre- and post-change regimes as well as for different states, i.e.,
$f_\theta^{(l)}(y) = \varphi(y-\mu_\theta^{(l)})$ ($\theta=\theta_0, \theta_1$, $l=1,2$), where $\varphi(y) =(2\pi)^{-1/2} \exp\set{-y/2}$ 
is density of the standard normal distribution. It is easily verified 
that the Kullback--Leibler number is finite. Condition \eqref{C2} in C2$^{'}$ has the form 
\[
 \abs{\sum_{i,l=1,2} \pi_\theta(i)  p_\theta(i,l) \int_{-\infty}^{\infty}  y^{r+1}  \varphi(y-\mu_\theta^{(l)}) \, \d y}  < \infty \quad \text{for}~\theta=\theta_0, \theta_1 ,
\]
which holds for all $r\ge 1$ due to the finiteness of all absolute positive moments of the normal distribution and the fact that
\[
\abs{\sum_{i,l=1,2} \pi_\theta(i)  p_\theta(i,l) \int_{-\infty}^{\infty}  y^{r+1}  \varphi(y-\mu_\theta^{(l)}) \, \d y} \le 
\sum_{i,l=1,2} \pi_\theta(i)  p_\theta(i,l) \int_{-\infty}^{\infty}  |y|^{r+1}  \varphi(y-\mu_\theta^{(l)}) \, \d y.
\]

Therefore, the Shiryaev detection rule is nearly optimal, minimizing asymptotically all positive moments of the detection delay.

\subsection{Example 3: Change of correlation in autoregression}

Consider the change of the correlation coefficient in the first-order autoregressive (AR) model
\begin{equation}\label{AR1}
Y_{n} = (a_{0}\Ind{n < \nu}+a(X_n)\Ind{n \ge \nu}) \,Y_{n-1}+\xi_{n}, \quad n \ge 1,
\end{equation}
where $X_{n}\in\{1,\dots,d\}$ is a $d$-state unobservable ergodic Markov chain and, conditioned on $X_n=l$, $a(X_n=l)= a_l$, $l=1,\dots,d$. The noise sequence $\{\xi_n\}$ is the i.i.d.\ standard Gaussian sequence,
$\xi_n\sim \Nc(0,1)$. Thus, the problem is to detect a change in the correlation coefficient of the Gaussian first-order AR process from the known value $a_0$ to a random value $a(X_n)\in\{a_1,\dots,a_d\}$ with 
possible switches between the given levels $a_1,\dots,a_d$ for $n\ge \nu$.

We assume that the transition matrix $[p(i,l)]$ is positive definite, i.e., $\det [p(i,l)]>0$ (evidently, it does not depend on $j=0,\infty$)  and that $|a_i|<1$ for $i=0,1,\dots,d$. 
The likelihood ratio for $Y_n$ given $\Yb^{n-1}_0$ and $X_n=l$  between the hypotheses $\Hyp_k$ and $\Hyp_\infty$ is
\[
\frac{p(Y_n|\Yb^{n-1}_0, X_n=l, \Hyp_k)}{p(Y_n| \Yb^{n-1}_0, \Hyp_\infty)} = \exp\set{\frac{1}{2} \brcs{(Y_n -a_0 Y_{n-1})^2- (Y_n -a_l Y_{n-1})^2}}, \quad n \ge k ,
\]
so that the likelihood ratio $\Lambda_n= p_0(Y_n|\Yb^{n-1}_0)/p_\infty(Y_n|\Yb^{n-1}_0)$ can be computed as
\[
\Lambda_n = \sum_{l=1}^d \exp\set{\frac{1}{2} \brcs{(Y_n -a_0 Y_{n-1})^2- (Y_n -a_l Y_{n-1})^2}} \Ps(X_n=l|\Yb^{n-1}_0),
\]
where using the Bayes rule, we obtain 
\[
\begin{aligned}
\Ps(X_n=l|\Yb^{n-1}_0) &= \sum_{i=1}^d p(i,l) \Ps(X_{n-1}=i |\Yb^{n-1}_0), 
\\
\Ps(X_n=l|\Yb^n_0) &= \frac{\Ps(X_n=l|\Yb^{n-1}_0) \exp\set{-\frac{1}{2} (Y_n -a_l Y_{n-1})^2}}{\sum_{i=1}^d \Ps(X_n=i|\Yb^{n-1}_0) \exp\set{-\frac{1}{2} (Y_n -a_i Y_{n-1})^2}}.
\end{aligned}
\]

The Markov chain $(X_n,Y_n)$ is $V$-uniformly ergodic with the Lyapunov function $V(y) = q (1+y^2)$, where $q\ \ge 1$, so condition C1 is satisfied. 
The condition C2 also holds. Indeed, if the change occurs from $a_0$ to the $i$-th component with probability 1, i.e., $\Ps(X_n=i)=1$ for $n \ge\nu$, then the Kullback--Leibler information number is equal to
\[
\begin{aligned}
&\int_{-\infty}^\infty \brcs{\int_{-\infty}^\infty [(y -a_0 x)^2- (y -a_i x)^2] \frac{1}{\sqrt{2\pi}} \exp\set{- \frac{1}{2} (y-x a_i)^2}d y} \sqrt{\frac{1-a_i^2}{2\pi}} 
\exp\set{- \frac{1-a_i^2}{2} x^2} \, \d x 
\\
&= \frac{(a_i-a_0)^2}{2(1-a_i^2)} .
\end{aligned}
\]
Hence, 
\[
0< \Kc <  \frac{(\max_{1 \le i \le d} a_i-a_0)^2}{2(1-\max_{1 \le i \le d} a_i^2)} < \infty.
\]
The condition \eqref{C22} in C2 holds for all $r \ge 1$ with $V(y) = q (1+y^2)$ since all moments of the Gaussian distribution are finite.


\appendix

\section{Auxiliary Results}\label{s:Proofs}

\subsection{A useful lemma}\label{ss:AppA1}
By $\lfloor x \rfloor$ we denote, as usual, the largest integer that is less than or equal to $x$. 

\begin{lemma}\label{Lem:A1}
For $0<\varepsilon <1$ and a small positive $\delta$, let $N_{\alpha,\varepsilon}= (1-\varepsilon)|\log\alpha|/(\Kc+c_\alpha+\delta)$. 
Let  the prior distribution $\{\omega_k^\alpha\}_{k\ge 0}$ satisfy condition \eqref{expon} with 
$c=c_\alpha\to 0$ as $\alpha\to 0$. Assume that the following condition holds
\begin{equation}\label{Pmaxgen}
\lim_{M\to\infty} \Ps_k\brc{\frac{1}{M} \max_{1 \le n \le M}S_{k+n-1}^k \ge (1+\varepsilon) \Kc} = 0 \quad \text{for all}~ \varepsilon >0~ \text{and all}~ k \ge 1.
\end{equation}  
Then
\begin{equation}\label{limprobA1}
\lim_{\alpha\to0} \sup_{T\in\classA} \Ps^\omega(0 \le T-\nu < N_{\alpha,\varepsilon}) =0 .
\end{equation}
\end{lemma}

\begin{proof}
Using an argument similar to that in the proof of Theorem~1 in \cite{TartakovskyVeerTVP05} that has lead to the inequality (3.11) in \cite{TartakovskyVeerTVP05}, 
we obtain that for any $T\in\classA$ 
\begin{equation}\label{ProbTle}
\Ps_k\brc{0 \le T-k < N_{\alpha,\varepsilon}} \le  p_k(\alpha,\varepsilon)  + \beta_k(\alpha,\varepsilon) ,
\end{equation}
where
\[
\beta_k(\alpha,\varepsilon) =\Ps_k\brc{\frac{1}{N_{\alpha,\varepsilon}} \max_{1 \le n \le N_{\alpha,\varepsilon}}S_{k+n-1}^k \ge (1+\varepsilon) \Kc}
\]
and 
\begin{equation}\label{Apkin}
p_k(\alpha,\varepsilon) \le  \exp\set{(1+\varepsilon) \Kc N_{\alpha,\varepsilon}-|\log\alpha| -\frac{\log \Ps(\nu \ge k+N_{\alpha,\varepsilon})}{ k+N_{\alpha,\varepsilon}}(k+N_{\alpha,\varepsilon})}. 
\end{equation}
Condition \eqref{expon} implies that for all sufficiently large $N_{\alpha,\varepsilon}$ (small $\alpha$), there is an arbitrary small $\delta=\delta_\alpha$ such that
\[
- \frac{\log \Ps(\nu \ge k+N_{\alpha,\varepsilon})}{k+N_{\alpha,\varepsilon}} \le c_\alpha + \delta_\alpha,
\]
where $c_\alpha, \delta_\alpha\to0$ as $\alpha\to0$. Using this inequality after simple algebra we obtain
\[
(1+\varepsilon) \Kc N_{\alpha,\varepsilon}-|\log\alpha| - \frac{\log \Ps(\nu \ge k+N_{\alpha,\varepsilon})}{k+N_{\alpha,\varepsilon}}(k+N_{\alpha,\varepsilon}) 
\le -\frac{\Kc}{\Kc+c_\alpha+\delta_\alpha} \varepsilon^2 |\log \alpha| + k(d_\alpha+\delta_\alpha)
\]
Hence, for all small $\alpha$, the following inequality holds:
\begin{equation}\label{UpperApk}
p_k(\alpha,\varepsilon) \le \exp\set{-\frac{\Kc}{\Kc+c_\alpha+\delta_\alpha}\varepsilon^2 |\log\alpha| + (c_\alpha+\delta_\alpha) k},
\end{equation}
where the right-hand side approaches zero as $\alpha\to0$ for 
\[
k\le K_{\alpha,\varepsilon}= \left \lfloor \frac{\Kc \,\varepsilon^3 |\log\alpha|}{(\Kc+c_\alpha+\delta_\alpha)(c_\alpha+\delta_\alpha)}\right \rfloor.
\]
Thus,
\[
\begin{aligned}
 \sup_{T\in \classA} & \Ps^\omega(0 < T-\nu < N_{\alpha,\varepsilon})  =\sum_{k=0}^\infty \omega_k^\alpha \sup_{T\in\classA} \Ps_k\brc{0 \le T-k < N_{\alpha,\varepsilon}}
 \\
 & \le  \Ps(\nu > K_{\alpha,\varepsilon}) + \sum_{k=0}^{K_{\alpha,\varepsilon}} \omega_k^\alpha \sup_{T\in\classA} \Ps_k\brc{0 \le T-k < N_{\alpha,\varepsilon}}
 \\
 &  \le \Ps(\nu > K_{\alpha,\varepsilon}) +  \sum_{k=0}^{K_{\alpha,\varepsilon}}  \omega_k^\alpha  \beta_k(\alpha,\varepsilon) 
 + \exp\set{-\frac{\Kc}{\Kc+c_\alpha+\delta_\alpha}\varepsilon^2 |\log\alpha| + (d_\alpha+\delta_\alpha) K_{\alpha,\varepsilon}} 
 \\
& \le  \Ps(\nu > K_{\alpha,\varepsilon}) +  \sum_{k=0}^{K_{\alpha,\varepsilon}}  \omega_k^\alpha  \beta_k(\alpha,\varepsilon) 
 + \exp\set{-\frac{\Kc}{\Kc+c_\alpha+\delta_\alpha}\varepsilon^2 (1-\varepsilon) |\log\alpha|} .
 \end{aligned}
 \]
By condition \eqref{expon}, 
\[
-\log \Ps(\nu > K_{\alpha,\varepsilon}) =O(|\log\alpha|) \to \infty \quad \text{as}~ \alpha\to0,
\]
so $\Ps(\nu > K_{\alpha,\varepsilon})\to0$ as $\alpha\to0$. By condition \eqref{Pmaxgen}, the second term goes to zero. The last term also vanishes as $\alpha\to0$
for all $0<\varepsilon<1$. Therefore, all three terms go to zero for all $0<\varepsilon <1$ and \eqref{limprobA1} follows.  
\end{proof}

\subsection{A Markov nonlinear renewal theory}\label{ss:AppA2}

In this subsection, we give a brief summary of the Markov nonlinear renewal theory developed in 
Fuh \cite{FuhAS04}. 
For the application of these results in this paper, we provide a simpler 
version which is more transparent. Note that here there is no description for the change point, we
use typical notations in Markov chains.

Abusing the notation a little bit we let $\{X_n, n\ge 0\}$ be a Markov chain 
on a general state space ${\cal X}$ with $\sigma$-algebra $\cal A$, which is irreducible
with respect to a maximal irreducibility measure on $({\cal X},\cal A)$
and is aperiodic. Let $S_n = \sum_{k=1}^n Y_k$ be the additive component,
taking values on the real line ${\bf R}$, such that
$\{(X_n,S_n), n\ge 0\}$ is a Markov chain on ${\cal X} \times {\bf R}$
with transition probability
\begin{equation}\label{8.4}
\begin{split}
&  \Ps\{(X_{n+1},S_{n+1}) \in A \times (B+s) | (X_n,S_n) = (x,s)\} \\
&= \Ps \{(X_1,S_1) \in A \times B | (X_0,S_0) = (x,0)\}
= \Ps(x,A \times B), 
\end{split}
\end{equation}
for all $x \in {\cal X},~ A \in {\cal A}$ and $B \in {\cal B}({\bf R})$ (Borel $\sigma$-algebra on 
${\bf R}$). The chain $\{(X_n,S_n), n \ge 0 \}$ is called a {\it Markov random walk}.
In this subsection, let $\Ps_\mu~(\E_{\mu})$ denote the probability (expectation) under
the initial distribution of $X_0$ being $\mu$. 
If $\mu$ is degenerate at $x$, we shall simply write $\Ps_x~(\E_x)$ instead of
$\Ps_\mu ~(\E_{\mu})$. We assume throughout this section that
there exists a stationary probability distribution $\pi$, $\pi(A) = \int \Ps(x,A) \, \d \pi(x)$
for all $A \in {\cal A}$ and $\E_{\pi} Y_1  >0$.

Let $\{Z_n = S_n + \eta_n, n\ge 0\}$ be a
perturbed Markov random walk in the following sense: $S_n$ is a Markov random walk, $\eta_n$ is
${\cal F}_n$-measurable, where ${\cal F}_n$ is the
$\sigma$-algebra generated by $\{(X_k,S_k), 0\le k \le n\}$, and $\eta_n$ is {\it slowly changing},
that is, $\max_{1 \le t \le n} |\eta_t|/n \to 0$ in probability.  For $\lambda \ge 0$ define
\begin{eqnarray}
 T = T_{\lambda} = \inf\{n \ge 1: Z_n > \lambda \},~~~\inf \varnothing
 = \infty.
\end{eqnarray}
It is easy to see that for all $\lambda > 0$, $T_\lambda < \infty$ with probability $1$.
This section concerns the approximations to the distribution of the overshoot
and the expected stopping time $\E_\mu T$ as the boundary tends to infinity.

We assume that the Markov chain $\{X_n,n \ge 0\}$ on a state space ${\cal X}$ is $V$-uniformly ergodic
defined as (\ref{3.11}). The following assumptions for Markov chains are used in this subsection.

A1. $\sup_x \big\{\frac{\E(V(X_1))}{V(x)} \big\} < \infty$,

A2. $\sup_x \E_x |Y_1|^2 < \infty$ and
$\sup_x \big\{\frac{\E(|Y_1|^r V(X_1))}{V(x)} \big\} < \infty$ for some $r \ge 1$.

A3. Let $\mu$ be an initial distribution of the Markov
chain $\{X_n, n \ge 0\}$. For some $r \ge 1$,
\begin{equation}
 \sup_{||h||_V \le 1} \abs{\int_{x \in {\cal X}} h(x) \E_x|Y_1|^{r} \, \d\mu(x) }  < \infty.
\end{equation}

A Markov random walk is called {\it lattice} with span $d > 0$ if $d$ is
the maximal number for which there exists a measurable function $\gamma:
{\cal X} \to [0,\infty)$ called the shift function, such that
$\Ps\{Y_1-\gamma(x) + \gamma(y) \in \{\cdots, -2d, -d, 0, d,
 2d,\cdots\}| X_0 = x, X_1 = y\} = 1$
for almost all $x,y\in {\cal X}$. If no such $d$ exists,
the Markov random walk is called {\it nonlattice}.  A lattice random walk
whose shift function $\gamma$ is identically 0 is called {\it arithmetic}.

To establish the Markov nonlinear renewal theorem, we shall make use of
(\ref{8.4}) in conjunction with the following extension
of Cramer's (strongly nonlattice) condition:
There exists $\delta > 0$ such that for all $m,n=1,2,\cdots$, $\delta^{-1} <m< n$,
and all $\theta \in R$ with $|\theta| \ge \delta$
\begin{eqnarray*}
\E_{\pi} | \E\{\exp(i\theta  (Y_{n-m}+\cdots+Y_{n+m}))|X_{n-m},\cdots,X_{n-1},X_{n+1},
\cdots,X_{n+m},X_{n+m+1}\}| \le e^{-\delta}.
\end{eqnarray*}

Let $\Ps_+^u(x,B \times R) = \Ps_x\{X_{\tau(0,u)} \in B \}$ for
$u \le \E_\pi Y_1$, denote the transition probability associated with the Markov
random walk generated by the ascending ladder variable
$S_{\tau(0,u)}$.  Here $\tau(0,u):=\inf\{n: S_n > 0\}. $  Under the $V$-uniform ergodicity condition
and $\E_\pi Y_1> 0$, a similar argument as on page 255 of Fuh and Lai \cite{FuhLai01}
yields that the transition probability $\Ps_+^u(x, \cdot \times R)$ has an invariant measure
$\pi_+^u$. Let $\E_{+}^u$ denote expectation when $X_0$ has the initial
distribution $\pi_+^u$. When $u=\E_\pi Y_1$, we denote $\Ps_+^{\E_\pi Y_1}$ as $\Ps_+$, and
$\tau_+ = \tau(0,\E_\pi Y_1)$. Define 
\begin{align}
\kappa &= \kappa_\lambda = Z_T - \lambda, \\
\kappa_+ &= \E_{+} S_{\tau_+}^2/2\E_{_+} S_{\tau_+}, \\
G(y) &= \frac{1}{\E_{+} S_{\tau_+}} \int_y^\infty \Ps_{+} \{S_{\tau_+} > s\} \, \d s ,~y \ge 0.
\end{align}

\begin{proposition}\label{P2}
Assume {\rm A1} holds, and {\rm A2}, {\rm A3} hold with $r=1$ and  $\E_\pi Y_1 \in (0,\infty)$. 
Let $\mu$ be an initial distribution of $X_0$. Suppose for every $\varepsilon > 0$ there is 
$\delta > 0$ such that
\begin{eqnarray}
&~& \lim_{n\to\infty} \Ps_\mu \brc{\max_{1\le j\le n \delta} |\eta_{n+j} - \eta_n| \ge \varepsilon}= 0.
\end{eqnarray}

If $Y_1$ does not have an arithmetic distribution under $\Ps_\mu$, then for any $y \ge 0$
\begin{eqnarray}
 \Ps_\mu \{X_T\in B, \kappa_\lambda > y\}  
= \frac{1}{\E_{+} S_{\tau_+}}\int_{x\in B} \, \d \pi_+(x) \int_y^\infty \Ps_{+} \{ S_{\tau_+} > s\} \, \d s +  o(1) ~~ \text{as}~\lambda \to \infty. 
\end{eqnarray}
In particular, $\Ps_\mu \{\kappa_\lambda > y\} = G(y) + o(1)$ as $\lambda\to\infty$ for any $r\ge 0$.
If, in addition, $(T-\lambda)/\sqrt{\lambda}$ converges in distribution to a random
variable $W$ as $\lambda \to \infty$, then
\begin{eqnarray}
\lim_{b_\lambda \to \infty} \Ps_\mu \{\kappa_\lambda > y,~ T_\lambda \ge \lambda + t\sqrt{\lambda}\} =    G(y)\Ps_+ \{W\ge t\},
\end{eqnarray}
for every real number $t$ with $P_+\{W=t\} = 0$.
\end{proposition}

To study the expected value of the nonlinear stopping times, we shall first give the regularity conditions on the nonlinear perturbation $\eta= \{\eta_n, n\ge 1\}$.
The process $\eta $ is said to be {\it regular} 
if there exists a random variable
$L$, a function $f(\cdot)$ and a sequence of random variables $U_n,~ n\ge 1$, such that
\begin{eqnarray}
&~& \eta_n = f(n)+U_n ~\mbox{ for~} n\ge L~ \mbox{ and~} \sup_{x \in {\cal X}} \E_x L <\infty, \label{8.13} \\
&~& \max_{1\le j\le\sqrt n} |f(n+j) - f(n)| \le K,~~~ K < \infty,  \\
&~& \big\{ \max_{1\le j\le n} |U_{n+j}|,~ n\ge 1 \big\}~
\mbox{is~uniformly~integrable},  \\
&~& n \sup_{x \in {\cal X}} \Ps_x \big \{\max_{0\le j\le n} U_{n+j} \ge \theta n \big\}
\to 0~~ \mbox{as~} n\to \infty ~{\rm for~all}~\theta > 0,\\
&~&\sum_{n=1}^\infty \sup_{x \in {\cal X}} \Ps_x\{ -U_n \ge wn\} < \infty~~~{\rm for~some}~
 0< w < \E_\pi Y_1,  
\end{eqnarray}
and there exists $0 < \varepsilon < 1$ such that
\begin{eqnarray}
 \sup_{x \in {\cal X}} \Ps_x \brc{ T_\lambda \leq \frac{\varepsilon \lambda}{\E_m Y_1} } =
 o({1}/{\lambda})~~~{\rm as}~\lambda \to \infty.  \label{8.17}
\end{eqnarray}

We need the following notation and definitions before we formulate the Markov Nonlinear Renewal Theorem (MNRT).
For a given Markov random walk $\{(X_n,S_n),n \ge 0\}$,
let $\mu$ be an initial distribution of $X_0$ and define
$\mu^*(B)=\sum_{n=0}^\infty \Ps_\mu\big(X_n\in B \big)$ on ${\cal A}$.
Let $g=\E(Y_1| X_0,X_1)$ and $\E_\pi|g|<\infty$.
Define operators ${\bf P}$ and ${\bf P}_\pi$ by
$({\bf P}g)(x)=\E_x g(x,X_1,Y_1)$ and ${\bf P}_\pi g
=\E_\pi g(X_0,X_1,Y_1)$ respectively,
and set $\overline{g}={\bf P}g$.  We shall consider solutions
$\Delta(x)=\Delta(x;g)$ of the Poisson equation
\begin{eqnarray}\label{Poi}
\big(I-{\bf P}\big)\Delta = \big(I-{\bf P}_\pi \big) \overline{g}~~~
\mu^*\mbox{-a.s.},~~~{\bf P}_\pi \Delta = 0,
\end{eqnarray}
where $I$ is the identity operator. Under conditions A1--A3,
it is known (cf. Theorem 17.4.2 of Meyn and Tweedie \cite{MeynTweedie09}) that the solution
$\Delta$ of (\ref{Poi}) exists and is bounded.

\begin{proposition}[MNRT]\label{P3}
Assume {\rm A1} holds, and {\rm A2}, {\rm A3} hold with $r=2$.
Let $\mu$ be an initial distribution such that $\E_\mu V(X_0) < \infty$. Suppose that
conditions {\rm (\ref{8.13})}--{\rm (\ref{8.17})} hold.
Then, as $\lambda\to\infty$,
\begin{equation}
\begin{array}{ll}
 \E_\mu T_\lambda = (\E_\pi Y_1)^{-1} \bigg(\lambda + {\E_{\pi_+} S_{\tau_+}^2}/{2 \E_{\pi_+}S_{\tau_+}}- f(\lambda/\E_\pi Y_1)
  -\E_{\pi}U \\
 ~~~~~~~~~~~~~~~~~~~~~~~ - \int\Delta(x) \, \d(\pi_+(x)  - \mu(x))\bigg)  + o(1). 
 \end{array}
\end{equation}

\end{proposition}




\end{document}